\documentclass{article}
\usepackage[utf8]{inputenc}
\usepackage{amssymb}
\usepackage{amsmath}
\usepackage{amsthm}
\usepackage{mathtools}
\usepackage{caption}
\usepackage{hyperref}
\usepackage{comment}
\usepackage{textcomp}
\usepackage[toc,page]{appendix}
\usepackage{algorithm, algorithmicx, algpseudocode}
\usepackage[square,numbers]{natbib}
\usepackage{amssymb}
\usepackage{amsmath}
\usepackage{enumerate}
\usepackage{subcaption}
\usepackage{textcomp}
\usepackage{url}
\usepackage[toc,page]{appendix}
\usepackage{algorithm, algorithmicx, algpseudocode}
\usepackage[square,numbers]{natbib}

\usepackage{graphicx}
\graphicspath{{plot/}}

\newtheorem{thm}{Theorem}
\newtheorem{lem}{Lemma}
\newtheorem{cor}{Corollary}
\newtheorem{rem}{Remark}
\newtheorem{defn}{Definition}

\DeclareMathOperator*{\argmax}{argmax}
\DeclareMathOperator*{\card}{card}
\DeclareMathOperator*{\diag}{diag}
\DeclareMathOperator*{\rank}{rank}

\begin{document}

\title{Deviation Maximization for Rank-Revealing QR Factorizations}

\author{Monica Dessole         \and
        Fabio Marcuzzi 
}
\date{June, 6th 2021}

\maketitle
{\centering Department of Mathematics ``Tullio Levi Civita''\\University of Padova\\Via Trieste 63, 35131 Padova, Italy\\ \footnotesize{e-mail: \verb#mdessole@math.unipd.it,marcuzzi@math.unipd.it#}\\}

\begin{abstract}
In this paper we introduce a new column selection strategy, named here ``Deviation Maximization", and apply it to compute rank-revealing QR factorizations as an alternative to the well known block version of the QR factorization with the column pivoting method, called QP3 and currently implemented in LAPACK's \texttt{xgeqp3} routine.
We show that the resulting algorithm, named QRDM, has similar rank-revealing properties of QP3 and better execution times. We present numerical test results on a wide data set of numerically singular matrices, which has become a reference in the recent literature.
\end{abstract}

\section{Introduction}

The rank-revealing QR (RRQR) factorization was introduced by Golub \cite{Gol65} and it is nowadays a classic topic in numerical linear algebra; for example, \citet{golub-libro} introduce Rank Revealing QR (RRQR) for least squares problems where the matrix has not full column rank: in such a case, a plain QR computation may lead to an $R$ factor in which the number of nonzeros on the diagonal does not equal the rank and the matrix $Q$ does not reveal the range nor the null space of the original matrix. Here, the SVD decomposition is the safest and most expensive solution method, while approaches based on a modified QR factorization can be seen as cheaper alternatives. Since the QR factorization is essentially unique once the column ordering is fixed, these techniques all amount to finding an appropriate column permutation. The first algorithm was proposed in \cite{BGo65} and it is referred as QR factorization with column pivoting (QRP).
It should be noticed that, if the matrix of the least squares problem has not full column rank, then there is an infinite number of solutions. We must resort to rank revealing techniques which identify a particular solution as ``special''. QR with column pivoting identify a particular \textit{basic} solution (with $r$ nonzeros, where $r$ is the rank), while biorthogonalization methods \cite{golub-libro}, identify the minimum $\ell_2$ solution. Rank-revealing decompositions can be used in a number of other applications \cite{Han99}.

The QR factorization with column pivoting works pretty well in practice, even if there are some examples in which it fails, see e.g. the Kahan matrix \cite{Kah66}. However, further improvements are possible, see e.g. \citet{Cha87} and \citet{Fos86}: the idea here is to identify and remove small singular values one by one. \citet{GEi96} introduced the Strong RRQR factorization, a stable algorithm for computing a RRQR factorization with a good approximation of the null space, which is not guaranteed by QR factorization with column pivoting. Both can be used as optional improvements to the QR factorization with column pivoting. Rank revealing QR factorizations were also treated in \cite{GKS76,CIp94,HPa92}.

Column pivoting makes it more difficult to achieve high performances in QR computation, see \cite{BHa92,BQu98,BQu98b,QSB98,Bis89}.
The state-of-the-art algorithm for computing RRQR, named QP3, is a block version \cite{QSB98} of the standard column pivoting and it is currently implemented in LAPACK \cite{lapack99}.
Other recent high-performance approaches are tournament pivoting \cite{DGG15} and randomized pivoting \cite{DGu17,XGL17,Mar15}. 
In this paper we present a column selection technique, that we call ``Deviation Maximization", and we propose it to derive QRDM, an alternative block algorithm to QP3 for computing RRQR factorizations.

The rest of this paper is organized as follows. In Section \ref{sec:DM} we motivate and present this novel column selection technique; in Section \ref{sec_RRQR_factorizations} we define the Rank-Revealing factorization, we review the QRP algorithm and then we introduce QRDM, a block algorithm for RRQR by means of Deviation Maximization furthermore, we give theoretical worst case bounds for the smallest singular value of the $R$ factor of the RRQR factorizations obtained with these two methods. In Section \ref{sec_implementation_issues} we discuss some fundamental issues regarding the implementation of QRDM; in Section \ref{sec:numerical_experiments} we compare QP3 and QRDM against a relevant database of singular matrices; finally, the paper concludes with Section \ref{sec:conclusions} and an Appendix with auxiliary results used in proofs.

\subsection{Notation} 
In what follows we denote by $\mathbb{O}$ and $\mathbb{I}$ the zero and identity matrices respectively with proper sizes (the former may even be rectangular). For any matrix $A$ of size $m \times n$, we denote by $[A]_{{I},{J}}$ the submatrix of $A$ obtained considering the entries with row and columns indices ranging in the sets ${I}$ and ${J}$, respectively. We make use of the so called ``colon notation", that is we denote by $[A]_{k:l,p:q}$ the submatrix of $A$ obtained considering the entries with row indices $k \leq i \leq l$ and column indices $p \leq j \leq q$. We use the shorthands $[A]_{{J}}$ and $[A]_{p:q}$ to indicate the submatrices $[A]_{1:m,{J}}$ and $[A]_{1:m, p:q}$ respectively, where only the column index is restricted.
We also denote the $(i,j)$-th entry as $a_{ij}$ ($a_{i,j}$) or $[A]_{ij}$ ($[A]_{i,j}$). The singular values of a matrix $A$ are denoted as
\begin{equation*}
    \sigma_{\max}(A) = \sigma_1(A) \geq \sigma_2(A) \geq \dots \geq \sigma_{\min}(A) = \sigma_{\min(m,n)}(A) \geq 0.
\end{equation*}
Given the vector norm $\Vert x \Vert_p = (|x_1|^p + \dots |x_n|^p)^{1/p}$, $p \geq 1$, we denote the family of $p$-norms as
\begin{equation*}
    \Vert A \Vert_p = \sup_{\Vert x\Vert_p = 1} \Vert Ax \Vert_p.
\end{equation*}
We denote the operator norm by $\Vert A \Vert_2 = \sigma_{\max}(A)$. When the context allows it, we drop the subscript on the $2$-norm. With a little abuse of notation, we define the max-norm of $A$ as $\Vert A \Vert_{\max} = \max_{i,j} |a_{ij}|$. Recall that the max-norm is not a matrix norm (it is not submultiplicative), and it should not be confused with the $\infty$-norm $\Vert A \Vert_{\infty} = \max_{i} \sum_j |a_{ij}|$. 

\section{Subset selection by Deviation Maximization} 
\label{sec:DM}
Consider an $m \times n$ matrix $A$ which has not full column rank, that is $\rank(A)= r < n$, and consider the problem of finding a subset of well conditioned columns of $A$. Before presenting a strategy to solve this problem, let us first introduce the notion of cosine matrix associated to a given matrix.
\begin{defn}
Let $C = (\mathbf{c}_{1}\ \dots \ \mathbf{c}_{k})$ be an $m\times k$ matrix whose columns $\mathbf{c}_{j}$ are non-null. Let $D$ be the diagonal matrix with entries $D_{ii} = \Vert \mathbf{c}_{i} \Vert\ ,\ 1\leq i \leq k$, then the \textbf{cosine matrix} associated to $C$ is defined as $\Theta = \Theta(C) = \left( C D^{-1} \right)^T C D^{-1} = D^{-1} C^T C D^{-1}$, and its entries are
\begin{equation}
    \theta_{ij} = \frac{\mathbf{c}_i^T \mathbf{c}_j}{\Vert \mathbf{c}_{i} \Vert\Vert \mathbf{c}_{j} \Vert} = \cos(\alpha_{ij}), \quad 1 \leq i,j \leq k.
    \label{eq:def_cosine_matrix}
\end{equation}
where $ \alpha_{ij} = \alpha(\mathbf{c}_i,\mathbf{c}_j)$ is the angle between the pair $\mathbf{c}_i,\mathbf{c}_j$. 
\end{defn}

It is immediate to see that the cosine matrix $\Theta$ is symmetric, it has only ones on the diagonal, and its entries range from $-1$ to $1$. 
The main idea behind Deviation Maximization is based on the following result.
\begin{lem}
Let $C = (\mathbf{c}_{1} \ \dots \ \mathbf{c}_{k})$ be an $m\times k$ matrix such that $\Vert\mathbf{c}_{1}\Vert = \max_j \Vert\mathbf{c}_{j}\Vert$, and let $\Theta$ be the associated cosine matrix. Suppose there exists $1 \geq \tau >0$ such that $\Vert \mathbf{c}_{j} \Vert \geq \tau \Vert\mathbf{c}_{1}\Vert$, for all $1 \leq j \leq k$. Moreover, suppose that $\Theta$ is a strictly diagonally dominant matrix with $\gamma= \min_i(1-\sum_{j\neq i} |\theta_{ij}|) >1-\tau^2>0$. Then $\sigma_{\min}(C) \geq \sqrt{\gamma + \tau^2-1}\ \Vert\mathbf{c}_{1}\Vert$.
\label{lem:tau_gamma}
\end{lem}
\begin{proof}
Let us first show that if $\Theta$ is a strictly diagonally dominant matrix, then the symmetric positive definite matrix $ C^TC = D \Theta D$, $D= \diag(\Vert \mathbf{c}_i\Vert)$ is strictly diagonally dominant. This follows from Lemma \ref{lem:SDD_rescaled_matrix}, in the Appendix, since $|\theta_{ii}| = 1$ for all $1\leq i \leq k$, and by hypothesis we have $\gamma>1-\tau^2$.

Applying the bound \eqref{eq:bound_sigma_min_SDD}, we have
\begin{equation*}
    \sigma_{\min}(C^TC) \geq (\gamma+ \tau^2 - 1)  \Vert\mathbf{c}_{1}\Vert^2  \quad \Rightarrow \quad \sigma_{\min}(C) \geq \sqrt{\gamma+ \tau^2 - 1 } \Vert\mathbf{c}_{1}\Vert. 
\end{equation*}
\qed\end{proof}

The result above shows quite clearly that the bound on the smallest singular value of $C$ depends on the norms of the column vectors and on the angles between each pair of such columns. 
This suggests to choose $k$ columns of $A$, namely those with indices ${J} = \left\{ j_1,\dots,j_k \right\}$, $k\leq r$, such that the columns of the corresponding submatrix $C=[A]_{{J}}$ have a large euclidean norm, i.e. larger than the length defined by $\tau$, and they are well separated, meaning that the cosines of the pairwise angles are bounded by a parameter $\delta$. The overall procedure is called Deviation Maximization and it is presented in Algorithm \ref{alg:DM}.
\begin{algorithm}[h]
\caption{Deviation Maximization ${J} = DM(A, \mathbf{u},\tau, \delta)$}
\label{alg:DM}
\begin{algorithmic}[1]
\State ${J} = \left\{j : j = \max \mathbf{u} \right\}$ \label{alg:DM_step_initializeJ}
\State ${I} = \left\{i\ :\ u_i \geq \tau \max \mathbf{u},\ i\neq j \right\} $ \label{alg:DM_step_candidate}
\State set $k_{\max} = \card({I})$ 
\State compute the cosine matrix $\Theta$ associated to $[A]_{{I}}$ \label{alg:DM_step_cosinematrix}
\For{$i \in {I}$}
\If{$ |\theta_{i,j}| < \delta, \forall j \in {J}$}{\ ${J} = {J} \cup \left\{ i \right\}$}
\EndIf
\EndFor
\end{algorithmic}
\end{algorithm}

More precisely, let us define the vector $\mathbf{u}$ containing the column norms of $A$, namely $\mathbf{u} = (u_i) = (\Vert  \mathbf{a}_i \Vert) $, for $i=1,\dots,n$.
The set ${J}$ of column indices is initialized at step \ref{alg:DM_step_initializeJ} with the column index corresponding to the maximum column norm, namely  
\begin{equation*}
{J} = \left\{j : j = \argmax \mathbf{u} \right\}.
\end{equation*}
At step \ref{alg:DM_step_candidate}, we identify a set of ``candidate'' column indices ${I}$ to add in ${J}$ by selecting those columns with a large norm with respect to the threshold $\tau$, that is 
\begin{equation}
{I} = \left\{i\ :\ u_i \geq \tau \max \mathbf{u},\ i\neq j \right\}, \label{eq:def_candidate_set}
\end{equation}
and we compute the cosine matrix associated to the corresponding submatrix $\Theta = \Theta([A]_{{I}})$ at step \ref{alg:DM_step_cosinematrix}.
With a loop over the indices of the candidate set ${I}$, an index $i\in {I}$ is inserted in ${J}$ only if the $i$-th column forms a large angle (i.e. the corresponding cosine is small) with the columns whose index is already in ${J}$. In formulae, we ask
\begin{equation*}
    |\theta_{i,j}| <  \delta, \qquad \text{for all }j \in {J}.
\end{equation*}
At the end of the iterations, we have ${J} = \left\{j_1,\dots,j_k \right\}$, with $1 \leq k \leq k_{max}$, where $k_{\max}$ is the cardinality of the candidate set $ {I}$, and we set $C = [A]_{{J}}$. 
Notice that the following choice of the parameter $\delta$, namely 
\begin{equation}
    \delta \leq \delta_{\max} := \frac{\tau^2}{k_{\max}-1}, \label{eq:delta_max}
\end{equation}
yields a submatrix $C=[A]_{{J}}$ that satisfies the hypotheses of Lemma \ref{lem:tau_gamma} for a fixed choice of parameters $\delta$ and $\tau$. Indeed, for every $j \in {J}$, this choice ensures
\begin{equation*}
    \sum_{\substack{i\in{J} \\ i\neq j}} |\theta_{ij}| < (k-1) \delta_{\max} = (k-1)\frac{\tau^2}{k_{\max}-1} \leq (k_{\max}-1) \frac{\tau^2}{k_{\max}-1} =\tau^2, 
\end{equation*}
and hence, the corresponding gap $\gamma$ of diagonal dominance of the cosine matrix $\Theta_{{J},{J}}$ of $C$ satisfies
\begin{equation*}
    \gamma = 1-\sum_{j\in{J}\\j\neq i} |\theta_{ij}| > 1-\tau^2,
\end{equation*}
fulfilling the hypotheses of Lemma \ref{lem:tau_gamma} and, therefore, we have
\begin{equation}
    \sigma_{\min}(C) = \sigma_{\min}([A]_{{J}}) \geq 
    \sqrt{ \gamma + \tau^2 -1} \max_i \mathbf{u} > 0.
    \label{eq:A_C_satisfies_lem_tau_gamma}
\end{equation}

Let us briefly comment the choice of the parameter $\tau$: on one hand, its value should be small in order to get a large candidate set ${I}$;
on the other hand, equation \eqref{eq:delta_max} shows that a large value of $k_{\max} = \card({I})$ implies a small value of $\delta_{\max}$, which likely yields a smaller $k = \card\left({J}\right)$. 
Notice that when $\delta_{\max}$ is equal (or close) to zero, only pairwise (nearly) orthogonal columns are accepted to be inserted in the set ${J}$, and it is unlikely to find matrices with pairwise (nearly) orthogonal columns in real world problems.

The procedure here presented exploits diagonal dominance in order to ensure linear independence. In practice, this often turns out to be a too strong condition to be satisfied, and as a result the number $k$ of columns found is usually way smaller than the rank $r$. Indeed, diagonal dominance is sufficient but obviously not necessary, so in the final algorithm we choose a weaker condition. 

The Deviation Maximization may be adopted as a block pivoting strategy in various algorithms that deal with subset selection. Actually, an effective and efficient strategy to choose $\delta$ and $\tau$ must be tied in practice to the properties of the specific algorithm that uses the Deviation Maximization as pivoting strategy. In this work we successfully apply the Deviation Maximization block pivoting to the problem of finding a Rank-Revealing QR decomposition, but e.g. the authors experimented also a preliminary version of this procedure in the context of active set methods, see \cite{DMV20DRNA, DMV20Ma}. 

\section{Rank-Revealing QR decompositions} 
\label{sec_RRQR_factorizations}
Let us introduce the mathematical formulation for the problem of finding a rank-revealing decomposition of a matrix $A$ of size $m \times n$.
We say that the matrix $A$ has numerical rank $1\leq r \leq \min(m,n)$ if $\sigma_{r+1}(A) \ll \sigma_{r}(A)$ and $\sigma_{r+1}(A) \approx \epsilon$, see \cite{CIp94}, where $\epsilon$ is the \textit{machine precision}. 
Let $\Pi$ denote a permutation matrix of size $n$, then we can compute
\begin{equation}
    A\Pi = QR = \left( Q_1\ Q_2 \right) \left( \begin{array}{cc}
        R_{11} & R_{12}  \\
        \mathbb{O} & R_{22} 
    \end{array} \right), 
    \label{eq:QRpivoting}
\end{equation}
where $Q$ is an orthogonal matrix of order $m$, $Q_1 \in m \times r$ and $Q_2 \in m \times (m-r)$,
$R_{11}$ is upper triangular of order $r$, $R_{12} \in r \times (n-r)$ and $R_{22}\in (m-r) \times (n-r)$. 
The QR factorization above is called \emph{rank-revealing} if
\begin{equation*}
    \sigma_{\min}(R_{11}) = \sigma_r(R_{11}) \approx \sigma_r(A),
\end{equation*}
or
\begin{equation*}
    \sigma_{\max}(R_{22}) = \sigma_1(R_{22}) \approx \sigma_{r+1}(A),
\end{equation*}
or both conditions hold. Notice that if $\sigma_{\min}(R_{11})\gg \epsilon$ and $\Vert R_{22} \Vert$ is small, then the matrix $A$ has numerical rank $r$, but the converse is not true. In other words, even if $A$ has $(\min(m,n)-r)$ small singular values, it is not guaranteed that any permutation $\Pi$ yields a small $\Vert R_{22} \Vert$.
It is easy to show that for any factorization like \eqref{eq:QRpivoting} the following relations hold 
\begin{align}
    & \sigma_{\min}(R_{11}) \leq \sigma_r(A), \label{eq:interlacing1}\\
    & \sigma_{\max}(R_{22}) \geq \sigma_{r+1}(A). \label{eq:interlacing2}
\end{align}
The proof follows is an easy application of the interlacing inequalities for singular values \cite{Tho72}, namely
\begin{equation*}
    \sigma_{k}(A) \geq \sigma_{k} (B) \geq \sigma_{k+r+s}(A), \quad k \geq 1,
\end{equation*}
which hold for any $(m-s)\times (n-r)$ submatrix $B$ of $A$. In fact we have
\begin{align*}
    & \sigma_{\min}(R_{11}) = \sigma_{\min}\left(\begin{array}{c}
         R_{11} \\
         \mathbb{O}
    \end{array}\right) = \sigma_{r}([Q^TA\Pi]_{1:m,1:r}) \leq \sigma_{r}(Q^TA\Pi) = \sigma_{r}(A), \\
    & \sigma_{\max}(R_{22}) = \sigma_{\max}(\mathbb{O}\ R_{22}) = \sigma_{1}([Q^TA\Pi]_{r+1:m,1:n}) \geq \sigma_{r+1}(Q^TA\Pi)
    = \sigma_{r+1}(A).
\end{align*}
We also used the invariance of the singular values under orthogononal transformations and under the insertion of a zero block, see equations (\ref{eq:sv_invariance_leftmult}-\ref{eq:sv_invariance_zeroblock}) in the Appendix.
Ideally, the best rank-revealing QR decomposition is obtained by the column permutation $\Pi$ which solves
\begin{equation}
    \max_{\Pi} \sigma_{\min} (R_{11}). 
    \label{eq:NPhard}
\end{equation}
However, the problem above clearly has a combinatorial nature. Therefore, algorithms that compute RRQR usually provide  (see, e.g. \cite{CIp94,HPa92}) at least one of the following bounds
\begin{align}
    \sigma_{\min}(R_{11}) &\geq \frac{\sigma_r(A)}{p(n)}, \label{eq:RRproperty_sigmamin}\\
    \sigma_{\max}(R_{22}) &\leq \sigma_{r+1}(A){q(n)}, \label{eq:RRproperty_sigmamax}
\end{align}
where $p(n)$ and $q(n)$ are low degree polynomials in $n$. These are worst case bounds and are usually not sharp. We provide a bound of type \eqref{eq:RRproperty_sigmamin} in Sec. \ref{sec_Worst-case bound}.

\subsection{QR factorization with column pivoting}
Let us introduce the QR factorization with column pivoting proposed by \citet{BGo65}, which can be labeled as a greedy approach in order to cope with the combinatorial optimization problem \eqref{eq:NPhard}.
Suppose at the $s$-th algorithmic step we have already selected $s<r$ well-conditioned columns of $A$, which are moved to the leading positions by the permutation matrix $\Pi^{(s)}$ as follows
\begin{equation}
     A \Pi^{(s)} = Q^{(s)} R^{(s)} = 
     Q^{(s)}
    \left(\begin{array}{cc}
        R_{11}^{(s)} & R_{12}^{(s)} \\
         & R_{22}^{(s)} 
    \end{array} \right), \label{eq:RRQR_step_s}
\end{equation}
where $R_{11}^{(s)}$ is an upper triangular block of size $s \times s$, and the blocks $ R_{12}^{(s)}$ and $ R_{22}^{(s)}$ have size $s\times(n-s)$ and $(m-s)\times(n-s)$ respectively. The block $R_{22}^{(s)}$ is what is left to be processed, and it is often called the ``trailing matrix". Let us introduce the following column partitions for $R_{12}^{(s)}$, $R_{22}^{(s)}$ respectively
\begin{align*}
    R_{12}^{(s)} &= \left( \mathbf{b}_1 \dots \mathbf{b}_{n-s} \right), \\
    R_{22}^{(s)} &= \left( \mathbf{c}_1 \dots \mathbf{c}_{n-s} \right).
\end{align*}
We aim at selecting, within the $n-s$ remaining columns, that column such that the condition number of the block $R_{11}^{(s+1)}$ is kept the largest possible. Formally, we would like to solve 
\begin{equation}
    \sigma_{\min}\left(\begin{array}{cc}
        R_{11}^{(s)} & \mathbf{b}_{j} \\
         & \mathbf{c}_{j}
    \end{array} \right) = 
    \max_{1 \leq i \leq n-s} \sigma_{\min}\left(\begin{array}{cc}
        R_{11}^{(s)} & \mathbf{b}_{i} \\
         & \mathbf{c}_{i} 
    \end{array} \right).
\label{eq:QRP_greedystep_k1}
\end{equation} 

Using the following fact
\begin{equation*}
    \sigma_{\min}\left(\begin{array}{cc}
        R_{11}^{(s)} & \mathbf{b}_{j} \\
         & \mathbf{c}_{j}
    \end{array} \right) 
    = 
    \sigma_{\min}\left(\begin{array}{cc}
        R_{11}^{(s)} & \mathbf{b}_{j} \\
         & \Vert \mathbf{c}_j \Vert
    \end{array} \right),
\end{equation*}
which is a simple consequence of the invariance of singular values under left multiplication by orthogonal matrices and the insertion of null rows, see \eqref{eq:sv_invariance_zeroblock} in the Appendix, and using the bound \eqref{eq:bound_smallest_sv}, we can approximate (up to a factor $\sqrt{s+1}$) the smallest singular value as 
\begin{equation*}
    \sigma_{\min}\left(\begin{array}{cc}
        R_{11}^{(s)} & \mathbf{b}_{j} \\
         & \mathbf{c}_{j}
    \end{array} \right) 
    \approx \min_{h} \left\Vert   \mathbf{e}_h^T \left(\begin{array}{cc}
        R_{11}^{(s)} & \mathbf{b}_{j} \\
         & \Vert \mathbf{c}_j \Vert
    \end{array} \right)^{-1} \right\Vert^{-1} ,
\end{equation*}
where $\mathbf{e}_h$ is the $h$-th element of the canonical basis of $\mathbb{R}^{s+1}$. 

Using this result, as argued in \cite{CIp94}, the maximization problem \eqref{eq:QRP_greedystep_k1} can be solved approximately by solving
\begin{equation*}
    j =  \argmax_{1 \leq i \leq n-n_s} \Vert \mathbf{c}_j \Vert \approx 
    \argmax_{1 \leq i \leq n-n_s} \sigma_{\min}\left(\begin{array}{cc}
        R_{11}^{(s)} & \mathbf{b}_{i} \\
         & \mathbf{c}_{i}
    \end{array} \right) .
\end{equation*} 
The resulting procedure is referred as QR factorization with column pivoting, and it is presented in Algorithm \ref{alg:QRpivoting}.
\begin{algorithm}[h]
\caption{QR with column pivoting $QRP(A)$}
\label{alg:QRpivoting}
\begin{algorithmic}[1]
\State initialize the vector $\mathbf{u}$ of squared norms 
\For{$s=0,\dots, n-1$}
\State $ j =  \argmax_i [\mathbf{u}]_{s+1:n}$
\State move the $j$-th index to the leading position on $A$ and $\mathbf{u}$
\State compute the Householder reflector $\mathbf{v}^{(s)}$ w.r.t. $[A]_{s:m,s}$
\State update the trailing matrix $[A]_{s:m,s+1:n} -= \mathbf{v}^{(s)}(\mathbf{v}^{(s)})^T\ [A]_{s:m,s+1:n}$
\State update the partial column norms $\mathbf{u}$
\EndFor
\end{algorithmic}
\end{algorithm}
This algorithm can be efficiently implemented since the column norms of the trailing matrix can be updated at each iteration instead of being recomputed from scratch. This can be done \cite{golub-libro} by exploiting the following property
\begin{equation*}
    Q \mathbf{a} =
    \begin{array}{cc}
    \left(\begin{array}{c}
         \beta \\
         \mathbf{c} 
    \end{array}\right)
    &
    \begin{array}{c}
        1 \\
        m-1
    \end{array}
    \end{array}  \Rightarrow \Vert \mathbf{a} \Vert^2 =  \Vert Q \mathbf{a} \Vert^2 = \beta^2 + \Vert \mathbf{c} \Vert^2,
\end{equation*}
which holds for any orthogonal matrix $Q$ and any vector $\mathbf{a}$ of order $m$.
Therefore, once defined the vector $\mathbf{u}^{(s)}$ whose entry $u^{(s)}_j$ is the $j$-th partial column norm of $A\Pi^{(s)}$, that is the norm of the subcolumn with row indices ranging from $m-n_s$ to $m$, and initialized $u^{(1)}_j = \Vert \mathbf{a}_j \Vert^2 $, with $1 \leq j \leq n$, we can perform the following update
\begin{equation}
    u^{(s+1)}_j = \begin{cases}
    \sqrt{(u^{(s)}_j)^2 - r_{sj}^2}, & s+1 \leq j \leq n, \quad 2 \leq s \leq n, \\
    0, & j < s+1, 
    \end{cases}
\label{eq:norm_update_colpivoting}
\end{equation}
where $r_{ij}$ is the entry of indices $(i,j)$ in $R^{(s)}$, $1\leq i\leq m$,  $1\leq j\leq n$. The partial column norm update allows to reduce the operation count from $\mathcal{O}(mn^2)$ to $\mathcal{O}(mn)$. Actually, the formula \eqref{eq:norm_update_colpivoting} cannot be applied as it is because of numerical cancellation, and it needs to  modified, see e.g. \cite{DBu08} for a robust implementation. A block version of Algorithm \ref{alg:QRpivoting} has been proposed \cite{QSB98}, and it is currently implemented in LAPACK's \texttt{xgeqp3} routine, that we will use in the numerical section for comparison.  

\begin{rem}{Geometric interpretation:} Introduce the following block column partitioning
    $R^{(s)} = (R_1^{(s)}\ R_2^{(s)})$, $Q^{(s)} = (Q_1^{(s)}\ Q_2^{(s)})$, and recall that we have 
    \begin{equation*}
     \mathcal{R}\left( Q_1^{(s)}\right) =\mathcal{R}\left( R_1^{(s)}\right) \quad \mathcal{R}\left( Q_2^{(s)}\right) =\mathcal{R}\left( R_1^{(s)}\right)^{\perp}.  
    \end{equation*}
    where $\mathcal{R}(B)$ denotes the subspace spanned by the columns of a matrix $B$. Every unprocessed column of $A$ rewrites as 
    \begin{equation*}
        \mathbf{a}_j = Q_1^{(s)}\mathbf{b}_{j-s} + Q_2^{(s)}\mathbf{c}_{j-s}, 
    \end{equation*}
    where $Q_1^{(s)}\mathbf{b}_{j-s}$ and $Q_2^{(s)}\mathbf{c}_{j-s}$ are the orthogonal projection of $\mathbf{a}_j$ on $\mathcal{R}\left( R_1^{(s)}\right)$ and $\mathcal{R}\left( R_1^{(s)}\right)^{\perp}$ respectively. The most linearly independent column $\mathbf{a}_i$ from the ones already processed can be seen as the one with the largest orthogonal projection of the complement on the subspace spanned by such columns, namely
    \begin{equation*}
        \max_{i\geq s} \left\Vert \mathcal{P}_{\mathcal{R}\left( R_1^{(s)}\right)^{\perp}} \mathbf{a}_i\right\Vert = \max_{i\geq 1} \left\Vert Q_2^{(s)}\mathbf{c}_i\right\Vert.
    \end{equation*}
    However, the matrix $Q^{(s)}$ is never directly available unless it is explicitly computed. We then settle for the the solution of the maximization problem $$\max_{i\geq 1} \Vert \mathbf{c}_i\Vert.$$
\end{rem}

\subsection{QR factorization with Deviation Maximization pivoting}
Consider the partial factorization in eq. \eqref{eq:RRQR_step_s}, and now suppose at the $s$-th algorithmic step we have already selected $n_s$, with $s \leq n_s<r$, well-conditioned columns of $A$, so that $R_{11}^{(s)}$ has size $n_s \times n_s$, while blocks $R_{12}^{(s)}$ and $R_{22}^{(s)}$ have size $n_s \times (n-n_s)$ and $(m-n_s) \times (n-n_s)$ respectively.
The idea is to pick $k_s$, with $n_{s+1} = n_s+k_s \leq r$, linearly independent and well-conditioned columns from the remaining $n-n_s$ columns of $A$, which are also sufficiently linearly independent from the $n_s$ columns already selected, in order to keep the smallest singular value of the $R_{11}$ block as large as possible.
We aim at selecting those columns with indices $j_1, \dots, j_{k_s}$ that solve
\begin{equation}
    \sigma_{\min}\left(\begin{array}{cccc}
        R_{11}^{(s)} & \mathbf{b}_{j_1} & \dots & \mathbf{b}_{j_{k_s}}\\
         & \mathbf{c}_{j_1} & \dots& \mathbf{c}_{j_{k_s}}
    \end{array} \right) = 
    \max_{1 \leq i_1,\dots,i_{k_s} \leq n-n_s} \sigma_{\min}\left(\begin{array}{cccc}
        R_{11}^{(s)} & \mathbf{b}_{i_1} & \dots & \mathbf{b}_{i_{k_s}}\\
         & \mathbf{c}_{i_1} & \dots& \mathbf{c}_{i_{k_s}}
    \end{array} \right).
    \label{eq:QRDM_greedystep_k1}
\end{equation}
Of course, this maximization problem has the same combinatorial nature as problem \eqref{eq:NPhard}, so we rather solve it approximately. We propose to approximate the indices $\left\{ {j_1}, \dots, {j_{k_s}} \right\} $ that solves problem \eqref{eq:QRDM_greedystep_k1} with the indices selected by means of the Deviation Maximization procedure presented in Algorithm \ref{alg:DM} applied to the trailing matrix $R_{22}^{(s)}$. For the moment, consider the parameter $\tau$ and $\delta$ fixed accordingly to equation \eqref{eq:delta_max}. More efficient choices will be widely discussed in Section \ref{sec:numerical_experiments}.
For sake of brevity, we will denote by $B = (\mathbf{b}_{j_1} \dots \mathbf{b}_{j_{k_s}})$ and $C = (\mathbf{c}_{j_1} \dots \mathbf{c}_{j_{k_s}})$ the matrices made up of the columns selected, and by $\bar{B}$ and $\bar{C}$ the matrices made up by the remaining columns.
The rest of the block update, which we detail below, proceeds in a way similar to the recursive block QR.
Let $ \tilde{Q}^{(s+1)}$ be an orthogonal matrix of order $(m-n_s)$ such that
\begin{equation}
    \left(\tilde{Q}^{(s)}\right)^T C =
     \left( \begin{array}{c}
        T \\
        \mathbb{O}
    \end{array}\right) \in \mathbb{R}^{(m-n_s)\times k_s},
    \label{eq:QR_of_C}
\end{equation}
where $T$ is an upper triangular matrix of order $k_s$. The matrix $\tilde{Q}^{(s+1)}$ is obtained as a product of $k_s$ Householder reflectors, that we represent by mean of the so-called compact WY form \cite{SVL89} as
\begin{equation*}
    \tilde{Q}^{(s)} = \mathbb{I} - Y^{(s)} W^{(s)} (Y^{(s)})^T,
\end{equation*}
where $Y^{(s)}$ is lower trapezoidal with $k_s$ columns and $W^{(s)}$ is upper triangular of order $k_s$. This allows us to carry out the update of the rest of trailing matrix, that is 
\begin{equation}
    \left(\tilde{Q}^{(s)}\right)^T \bar{C} =  \left(\begin{array}{c}
         \bar{T} \\
         R_{22}^{(s+1)}
    \end{array}\right)  \in \mathbb{R}^{(m-n_s)\times (n-n_s-k_s)},
    \label{eq:QR_of_C_applied_to_Cp}
\end{equation}
by means of BLAS-3 kernels, for performance efficiency.
Denoting by $\tilde{\Pi}^{(s)}$ a permutation matrix that moves columns with indices $j_1,\dots,j_{k_s}$ to the current leading positions, we set $\Pi^{(s+1)} = \Pi^{(s)}\tilde{\Pi}^{(s)}$ and 
\begin{equation*}
    Q^{(s+1)} = Q^{(s)}
    \left(\begin{array}{cc}
        \mathbb{I} &  \\
         & \tilde{Q}^{(s)}
    \end{array}\right) \in \mathbb{R}^{m\times m},
\end{equation*}
then the overall factorization of $A \Pi^{(s+1)}$ takes the form
\begin{equation}
     Q^{(s)}
    \left(\begin{array}{ccc}
        R_{11}^{(s)} & B & \bar{B} \\
         &  C & \bar{C}
    \end{array} \right)=  Q^{(s+1)}
    \left(\begin{array}{ccc}
        R_{11}^{(s)} & B & \bar{B} \\
          & T & \bar{T} \\
         &   & R_{22}^{(s+1)}
    \end{array} \right), 
    \label{eq:block_partition_sp1}
\end{equation}
where, for the successive iteration, we set
\begin{equation*}
    \begin{aligned}
    & R_{11}^{(s+1)} =         \left(\begin{array}{cc}
        R_{11}^{(s)} & B \\
          & T \\
    \end{array} \right) \in \mathbb{R}^{n_{s+1}\times n_{s+1}}, \\
        & R_{12}^{(s+1)} =         \left(\begin{array}{c}
         \bar{B} \\
         \bar{T}
    \end{array} \right) \in \mathbb{R}^{n_{s+1}\times(n-n_{s+1})}, \\
    \end{aligned}
\end{equation*}
with $n_{s+1} = n_s+k_s$. The resulting procedure is called QR factorisation with Deviation Maximization pivoting and it presented in Algorithm \ref{alg:QRDM}. 
\begin{algorithm}[h]
\caption{QR with Deviation Maximization pivoting $QRDM(A, \tau, \delta)$}
\label{alg:QRDM}
\begin{algorithmic}[1]
\State set $n_s=0$ and initialize the vector $\mathbf{u}$ of squared norms 
\While{$n_s<n$}
\State $\left\{j_1,\dots,j_{k_s}\right\} = DM([A]_{n_s+1:m,n_s:n}, [\mathbf{u}]_{n_s+1:n},\tau, \delta)$
\State move columns $ \left\{j_1,\dots,j_{k_s}\right\}$ to the leading positions of $A$ and $\mathbf{u}$
\For{$l=1,\dots, k_s$}
\State compute the Householder reflector $\mathbf{v}^{(n_s+l)}$ w.r.t. $[A]_{n_s+l:m,n_s+l}$
\State update the remaining columns $[A]_{n_s+l:m,n_s+l+1:n_s+k_s} -= \mathbf{v}^{(n_s+l)}(\mathbf{v}^{(n_s+l)})^T\ [A]_{n_s+l:m,n_s+l+1:n_s+k_s}$
\EndFor
\State compute the compact WY representation of $\mathbf{v}^{(n_s+1)},\dots, \mathbf{v}^{(n_s+k_s)}$ \label{alg:step_QRDM_WYrepr}
\State block update $[A]_{n_s+1:m,n_s+k_s+1:n} -= Y^{(s)}\ (W^{(s)})^T\ (Y^{(s)})^T\ [A]_{n_s+1:m,n_s+k_s+1:n}$
\State update the partial column norms $\mathbf{u}$
\State $n_s = n_s+k_s$
\EndWhile
\end{algorithmic}
\end{algorithm}

Last, we point out that the partial column norms can be updated at each iteration also in this case with some straightforward changes of equation \eqref{eq:norm_update_colpivoting}, namely   
\begin{equation*}
    u^{(s+1)}_j =
    \begin{cases}
    \sqrt{(u^{(s)}_j)^2 - \displaystyle \sum_{l=n_s}^{n_{s+1}} r_{lj}^2}, & n_{s+1} < j \leq n, \quad n_{s+1} \leq n, \\
    0, & j \leq n_{s+1}.
    \end{cases}
\end{equation*}

The QRP algorithm has the particular feature that the diagonal elements of the final upper triangular factor $R$ are monotonically non increasing in modulus. This property cannot be guaranteed by the QRDM algorithm, as by other recently proposed pivoting strategies \cite{DGG15}. In practice, there are small fluctuations around a non-increasing trend. 

\subsection{Worst-case bound on the smallest singular value}
\label{sec_Worst-case bound}

Let us denote by $\bar{\sigma}^{(s)}$ the smallest singular value of the computed $R_{11}^{(s)}$ block at step $s$, that is
\begin{equation*}
    \bar{\sigma}^{(s)} = \sigma_{\min}\left(R_{11}^{(s)} \right).
\end{equation*}
Notice that it corresponds exactly to the $s$-th singular value of $R_{11}^{(s)}$ computed with the standard column pivoting, while it corresponds to the $n_s$-th singular value when $R_{11}^{(s)}$ is computed with the Deviation Maximization.
Let us first report from \cite{CIp94} an estimate of $\bar{\sigma}^{(s+1)}$ for QRP.
\begin{thm}
Let $R_{11}^{(s)}$ be the upper triangular factor of order $s$ computed by QRP. We have
\begin{equation*}
    \bar{\sigma}_{s+1} \geq \sigma_{s+1}(A) \frac{\bar{\sigma}_{s}}{\sigma_1(A)}\frac{1}{\sqrt{2(n-s)(s+1)}}.
\end{equation*}
\end{thm}

Before coming to the main result, we introduce the following auxiliary Lemma.

\begin{lem}\label{lem:bound_sigma_T} 
With reference to the notation used for introducing the block partition in \eqref{eq:block_partition_sp1}, we have
\begin{equation}
    \sigma_{\min}(T) \geq \frac{\sqrt{\gamma+ \tau^2 - 1 }}{\sqrt{n-n_{s+1}+1}} \sigma_{n_{s+1}}(A).
\end{equation}
\end{lem}
\begin{proof}
Consider following column partitions $T = (\mathbf{t}_1 \dots \mathbf{t}_k)$, $\bar{T} = (\mathbf{t}_{k+1} \dots \mathbf{t}_{n-n_s})$, $R_{22}^{(s+1)} = (\mathbf{r}_{k+1}\dots \mathbf{r}_{n-n_s})$, and set $\mathbf{r}_j = \mathbf{0}$, for $1\leq j \leq k$. Moreover, let $T=\left\{t_{i,j}\right\}$, with ${1\leq i \leq j \leq k}$, and $\bar{T}=\left\{t_{i,j}\right\}$ with ${1\leq i \leq k, 1 \leq j \leq n-n_s}$. 
First, notice that by eq. \eqref{eq:interlacing2} we have
\begin{equation*}
     \left\Vert \begin{array}{cc} t_{k,k} & t_{k,k+1}, \dots,t_{k,n-n_s} \\ \mathbf{0} & R_{22}^{(s+1)} \end{array} \right\Vert \geq \sigma_{n_{s+1}}(A).
\end{equation*}
From eq. \eqref{eq:bound_largest_sv}, we have
\begin{equation*}
    \left\Vert \begin{array}{cc} t_{k,k} & t_{k,k+1}, \dots,t_{k,n-n_s} \\ \mathbf{0} & R_{22}^{(s+1)} \end{array} \right\Vert^2 
    \leq 
    (n-n_{s+1}+1) \max\left\{t_{k,k}^2, \max_{j\geq k+1}\left(\Vert \mathbf{r}_{j}\Vert^2 + t_{k,j}^2\right)\right\}.
\end{equation*}
Since $t_{k,j}^2 \leq \Vert \mathbf{t}_j \Vert^2$, for all $1 \leq j \leq n-n_s$, and computing the maximum on a larger set of indices we have
\begin{align*}
    \max\left\{t_{k,k}^2, \max_{j\geq k+1}\left(\Vert \mathbf{r}_{j}\Vert^2 + t_{k,j}^2\right)\right\}
    &\leq
    \max\left\{\Vert \mathbf{t}_k \Vert^2, \max_{j\geq k+1}\left(\Vert \mathbf{r}_{j}\Vert^2 + \Vert \mathbf{t}_j \Vert^2 \right)\right\} \\
    &\leq 
    \max_{j\geq 1}\left(\Vert \mathbf{r}_j\Vert^2 + \Vert \mathbf{t}_j\Vert^2\right).
\end{align*}
From equations (\ref{eq:QR_of_C}-\ref{eq:QR_of_C_applied_to_Cp}), for all $1 \leq j \leq n-n_s$, we have 
\begin{equation*}
    \Vert \mathbf{c}_j\Vert^2 = \Vert \mathbf{r}_j\Vert^2 + \Vert \mathbf{t}_j\Vert^2, 
\end{equation*}
and, finally, since $\Vert \mathbf{t}_1\Vert^2  = \Vert \mathbf{c}_1\Vert^2 = \max_j \Vert \mathbf{c}_j\Vert^2$ and by using Lemma \ref{lem:tau_gamma}, we get
\begin{equation*}
    \left\Vert \begin{array}{cc} t_{k,k} & t_{k,k+1}, \dots,t_{k,n-n_s} \\ \mathbf{0} & R_{22}^{(s+1)} \end{array} \right\Vert^2
    \leq (n-n_{s+1}+1) \Vert \mathbf{c}_1\Vert^2  \leq \frac{n-n_{s+1}+1}{{\gamma+ \tau^2 - 1 }} \sigma_{\min}^2(C).
\end{equation*}
We can conclude by noticing that $\sigma_{\min}(T)=\sigma_{\min}(C)$, since the two matrices differ by a left multiplication by an orthogonal matrix.
\qed\end{proof}

By the interlacing property of singular values, we have 
\begin{equation*}
    \bar{\sigma}^{(s+1)} \leq 
    \min \left\{ \bar{\sigma}^{(s)},\sigma_{\min}
    \left(\begin{array}{c}
         B  \\
         T 
    \end{array}\right) \right\},
\end{equation*}
thus the bounds on $\bar{\sigma}^{(s)}$ and $\sigma_{\min}(T)$ are, by themselves, not a sufficient condition.
Let us introduce the following result, which provides a bound of type \eqref{eq:RRproperty_sigmamin} for QRDM.
\begin{thm}
Let $R_{11}^{(s)}$ be the upper triangular factor of order $n_s$ computed by QRDM. We have
\begin{equation*}
    \bar{\sigma}^{(s+1)} \geq \sigma_{n_{s+1}}(A)  \frac{\bar{\sigma}^{(s)}}{\sigma_1(A)} \frac{1}{\sqrt{2(n-n_{s+1})n_{s+1}}}  \frac{\sqrt{\gamma+ \tau^2 - 1 }}{k^2 n_s}.
\end{equation*}
\end{thm}
\begin{proof}
Let us drop the subscript and the superscript on the inverse of $R_{11}^{(s)}$ and its inverse $\left(R_{11}^{(s)}\right)^{-1}$, which will be denoted as $R$ and $R^{-1}$ respectively.
Then, the inverse of matrix $R_{11}^{(s+1)}$ is given by
\begin{equation*}
   \left(R_{11}^{(s+1)}\right)^{-1} = \left(\begin{array}{cc}
        R^{-1} & -R^{-1}BT^{-1} \\
         & T^{-1} 
    \end{array} \right).
\end{equation*}
Let us introduce the following partitions into rows
\begin{equation*}
  F = R^{-1}BT^{-1} = \left(\begin{array}{c}\mathbf{f}_1^T\\ \vdots\\ \mathbf{f}_{n_s}^T \end{array}\right),\quad
  R^{-1} = \left(\begin{array}{c}\mathbf{g}_1^T\\ \vdots\\ \mathbf{g}_{n_s}^T \end{array}\right), \quad
  T^{-1} = \left(\begin{array}{c}\mathbf{h}_1^T\\ \vdots\\ \mathbf{h}_{k}^T \end{array}\right).
\end{equation*}
The idea is to use eq. \eqref{eq:bound_smallest_sv}, that is
\begin{equation*}
    \bar{\sigma}^{(s+1)} \leq \min_h \left\Vert \mathbf{e}_h^T\left( \begin{array}{cc}
        R^{-1} & F \\
         & T^{-1} 
    \end{array} \right) \right\Vert^{-1} \leq \sqrt{n_{s+1}}\sigma_{\min}\bar{\sigma}^{(s+1)} ,
\end{equation*}
to estimate the minimum singular value up to a factor $\sqrt{n_{s+1}}$.
For $1 \leq h \leq n_{s+1} $ we have
\begin{equation*}
    \left\Vert \mathbf{e}_h^T\left( \begin{array}{cc}
        R^{-1} & F \\
         & T^{-1} 
    \end{array} \right) \right\Vert^2 = 
    \begin{cases}
    \Vert \mathbf{g}_h\Vert^2 + \Vert \mathbf{f}_h\Vert^2, & h\leq n_s,\\
    \Vert \mathbf{h}_{h-n_s}\Vert^2, & h> n_s.
    \end{cases}
\end{equation*}
We can bound $ \Vert \mathbf{g}_h\Vert $ using  eq. \eqref{eq:bound_smallest_sv} again, which gives
\begin{equation*}
 \bar{\sigma}^{(s)} \leq \min_h \left( \left\Vert \mathbf{g}_h \right\Vert^{-1} \right) \leq \sqrt{n_s} \bar{\sigma}^{(s)} .
\end{equation*}
In particular, for every $1\leq h\leq n_s$, we get
\begin{equation*}
    \bar{\sigma}^{(s)}  \leq \min_h \left( \left\Vert \mathbf{g}_h \right\Vert^{-1} \right) \leq \left\Vert \mathbf{g}_h \right\Vert^{-1},
\end{equation*}
and thus we have
\begin{equation*}
      \left\Vert \mathbf{g}_h \right\Vert \leq = \frac{1}{\bar{\sigma}^{(s)} } = \frac{1}{\sigma_{\min}(R)} = \sigma_{\max}(R^{-1}) = \Vert R^{-1} \Vert.
\end{equation*}
Similarly, we can bound $ \Vert\mathbf{h}_{h-n_s}\Vert $ by $\Vert T^{-1} \Vert$.
Let us now concentrate on bounding $ \Vert \mathbf{f}_h\Vert$. We have
\begin{equation*}
\begin{aligned}
    \Vert \mathbf{f}_h\Vert_2 &\leq  \Vert \mathbf{f}_h\Vert_1 = \sum_{l=1}^{k} \left| f_{hl} \right| = \sum_{l=1}^{k} \left| \sum_{i=1}^{k} [R^{-1}B]_{hi}[T^{-1}]_{il} \right|  \\
     &= \sum_{l=1}^{k} \left| \sum_{i=1}^{k} \sum_{j=1}^{n_s} [R^{-1}]_{hj}[B]_{ji}[T^{-1}]_{il} \right|  \\
     &\leq \sum_{l=1}^{k} \sum_{i=1}^{k} \sum_{j=1}^{n_s} \left| [R^{-1}]_{hj} \right|\ \left| [B]_{ji} \right|\ \left| [T^{-1}]_{il} \right| \\
     &\leq  \sum_{l=1}^{k} \sum_{i=1}^{k} \sum_{j=1}^{n_s}  \left\Vert R^{-1} \right\Vert_{\max}  \left\Vert B \right\Vert_{\max}  \left\Vert T^{-1} \right\Vert_{\max} \\
     &= k^2 n_s \left\Vert R^{-1} \right\Vert_{\max}  \left\Vert B \right\Vert_{\max}  \left\Vert T^{-1} \right\Vert_{\max}  \\
     & \leq k^2 n_s \left\Vert R^{-1} \right\Vert  \left\Vert B \right\Vert  \left\Vert T^{-1} \right\Vert \\
     &= \frac{k^2 n_s}{\bar{\sigma}^{(s)}}  \left\Vert B \right\Vert  \left\Vert T^{-1} \right\Vert,
  \end{aligned}
\end{equation*}
where we use the following facts $\Vert \mathbf{x}\Vert_2 \leq \Vert \mathbf{x}\Vert_1$, and $\left\Vert A \right\Vert_{\max} \leq \left\Vert A \right\Vert$, see \eqref{eq:bound_max_norm}.
Moreover, we can write
\begin{equation*}
\begin{aligned}
&\Vert \mathbf{g}_h\Vert^2 + \Vert \mathbf{f}_h\Vert^2 \leq \frac{1}{(\bar{\sigma}^{(s)})^2} + \frac{k^4 n_s^2}{(\bar{\sigma}^{(s)})^2} \left\Vert B \right\Vert^2  \left\Vert T^{-1} \right\Vert^2  \\
&= \frac{\sigma_{\min}^2(T) + k^4 n_s^2 \left\Vert B \right\Vert^2 }{(\bar{\sigma}^{(s)}\sigma_{\min}(T))^2} \leq \frac{\Vert T\Vert^2 + k^4 n_s^2 \left\Vert B \right\Vert^2 }{(\bar{\sigma}^{(s)}\sigma_{\min}(T))^2} \\
& \leq \frac{ 2k^4 n_s^2 }{(\bar{\sigma}^{(s)}\sigma_{\min}(T))^2} \max\left\{ \Vert T\Vert^2, \Vert B\Vert^2\right\} \leq \frac{2 k^4 n_s^2 }{(\bar{\sigma}^{(s)}\sigma_{\min}(T))^2} \Vert A\Vert^2,
  \end{aligned}
\end{equation*}
where, in the last inequality, we used the interlacing property and the invariance under matrix transposition of the singular values. In fact
\begin{equation*}
    \sigma_{1}(A) \geq \sigma_{1}\left(\begin{array}{c}
         B \\
         T 
    \end{array} \right) = \sigma_{1}\left(B^T\ T^T \right) \geq \max\left\{\sigma_{1}(B), \sigma_{1}(T)\right\}.
\end{equation*}
Hence, we get
\begin{equation*}
\begin{aligned}
    \frac{1}{\sqrt{\Vert \mathbf{g}_h\Vert^2 + \Vert \mathbf{f}_h\Vert^2}} \geq \frac{\bar{\sigma}^{(s)}\sigma_{\min}(T)}{\sqrt{2} k^2 n_s \sigma_1(A)}.
\end{aligned}
\end{equation*}
If $\bar{\sigma}^{(s)}$ is a good approximation of $\sigma_{n_s}(A)$, we can suppose that $\bar{\sigma}^{(s)}/\sigma_{n_s}(A) \approx 1$, and we can write
\begin{equation*}
\begin{aligned}
    \sqrt{n_{s+1}}  \bar{\sigma}^{(s+1)} &\geq \min \left\{ \min_h \Vert \mathbf{h}_h\Vert^{-1}, \min_h \frac{1}{\sqrt{\Vert \mathbf{g}_h\Vert^2 + \Vert \mathbf{f}_h\Vert^2}} \right\} \\
    &\geq \min \left\{ 1, \frac{\bar{\sigma}^{(s)}}{\sqrt{2} k^2 n_s \sigma_1(A)} \right\}\sigma_{\min}(T) \\
    & = \frac{\bar{\sigma}^{(s)}}{\sqrt{2} k^2 n_s \sigma_1(A)} \sigma_{\min}(T).
\end{aligned}
\end{equation*}
Finally, using Lemma \ref{lem:bound_sigma_T}, we get
\begin{equation*}
    \bar{\sigma}^{(s+1)} \geq \sigma_{n_{s+1}}(A)  \frac{\bar{\sigma}^{(s)}}{\sigma_1(A)} \frac{1}{\sqrt{2(n-n_{s+1})n_{s+1}}}  \frac{\sqrt{\gamma+ \tau^2 - 1 }}{k^2 n_s},
\end{equation*}
which is the desired bound.
\qed\end{proof}

This shows that even if the leading $n_s$ columns have been carefully selected, so that $\bar{\sigma}^{(s)}$ is an accurate approximation of $\sigma_{n_s}(A)$, there could be a potentially dramatic loss of accuracy in the estimation of the successive block of singular values, namely $\sigma_{n_s+1}(A),\dots,\sigma_{n_{s+1}}(A)$, just like for the standard column pivoting. 
In fact, it is well known that failure of QRP algorithm may occur (one such example is the Kahan matrix \cite{Kah66}), as well as for other greedy algorithms, but it is very unlikely in practice.

\subsection{Termination criteria}
In principle, both QRP and QRDM reveal the rank of a matrix. In finite arithmetic we have
\begin{equation}
    \left(\begin{array}{cc}
        \hat{R}_{11}^{(s)} & \hat{R}_{12}^{(s)} \\
         & \hat{R}_{22}^{(s)} 
    \end{array} \right), \label{eq:RRQR_step_s_rounded}
\end{equation}
where $\hat{R}^{(s)}_{ij}$ is the block $R^{(s)}_{ij}$ computed in finite representation, for $i=1,2,j=2$. If the block $ \hat{R}^{(s)}_{22} $ is small in norm, then it is reasonable to say that the matrix $A$ has rank $n_s$, where $n_s$ is the order of the upper triangular block $\hat{R}_{11}^{(s)}$. \citet{golub-libro} propose the following termination criterion
\begin{equation}
    \left\Vert  \hat{R}^{(s)}_{22}  \right\Vert \leq \epsilon_1 \left\Vert A  \right\Vert, \label{eq:stop_crit_golub}
\end{equation}
where $\epsilon_1$ is a parameter depending on the machine precision $\epsilon$. Notice that even if an $R_{22}$ block with small norm implies numerical rank-deficiency, the converse is not true in general: an example is the Kahan matrix \cite{Kah66}. Since the $2$-norm is not directly available, we make use of the inequalities \eqref{eq:bound_largest_sv}. Let us write the column partition $ \hat{R}^{(s)}_{22} = (\hat{\mathbf{c}}_1\ \dots \ \hat{\mathbf{c}}_{n-n_s})$. We have 
\begin{equation*}
    \left\Vert  \hat{R}^{(s)}_{22}  \right\Vert \leq \sqrt{n-n_s} \max_i  \left\Vert  \hat{\mathbf{c}}_i \right\Vert, \quad
    \max_i  \left\Vert  \mathbf{a}_i \right\Vert \leq  \left\Vert  A  \right\Vert. 
\end{equation*}
Therefore, the stopping criterion \eqref{eq:stop_crit_golub} holds if 
\begin{equation}
    \sqrt{n-n_s} \max_i  \left\Vert  \hat{\mathbf{c}}_i \right\Vert \leq \epsilon_1 \max_i  \left\Vert  \mathbf{a}_i \right\Vert.
    \label{eq:practical stopping criterion}
\end{equation}
Notice that the contrary does not hold. In Section \ref{sec:numerical_experiments} we test this practical stopping criterion \eqref{eq:practical stopping criterion} and discuss the following two choices:
\begin{align}
    \epsilon_1 &=\epsilon \ n, \label{eq:practical stopping criterion_choice1}\\
    \epsilon_1 &=\epsilon \sqrt{n}. \label{eq:practical stopping criterion_choice2}
\end{align}

\section{Implementation issues}
\label{sec_implementation_issues}

In this section we discuss implementation aspects of the QRDM procedure. In particular, we address the following issues
\begin{enumerate}
    \item the practical computation of the candidate set ${I}$ defined in \eqref{eq:def_candidate_set};
    \item the efficient computation of the cosine matrix $\Theta$ defined in \eqref{eq:def_cosine_matrix};
    \item make a less restrictive choice of the parameters $\tau$ and $\delta$ \eqref{eq:delta_max}, without affecting the robustness of the computed $QR$; 
    \item the structure of the pivoting, which has a significant impact on the cost of the algorithm.
 \end{enumerate}
Let us first focus on some details of the implementation of the Deviation Maximization presented in Algorithm \ref{alg:DM}.

The candidate set ${I}$ can be computed with a fast sorting algorithm, e.g. quicksort, applied to the array of partial column norms. 

The most expensive operation in Algorithm \ref{alg:DM} is the computation of the cosine matrix in step \ref{alg:DM_step_cosinematrix}. If we write the matrix $A$ by columns $A = ( \mathbf{c}_1\ \dots \ \mathbf{c}_n)$, then the cosine matrix $\Theta$ has entries $\theta_{ij} = \mathbf{c}_i^T \mathbf{c}_j \Vert \mathbf{c}_i \Vert^{-1} \Vert \mathbf{c}_j \Vert^{-1}$, for $i,j \in {I}$. Therefore we have
\begin{equation*}
    \Theta = D^{-1} [A]^T_{ {I}} [A]_{ {I}} D^{-1},
\end{equation*}
where $D = \diag(\Vert \mathbf{c}_j \Vert)$, with $j \in {I}$.
The matrix $\Theta$ is symmetric, thus we only need its upper (lower) triangular part. This can be computed in two ways
\begin{enumerate}[(i)]
    \item we first form the product $U_1 = [A]_{ {I}} D^{-1}$, and then we compute $\Theta = U_1^T U_1$;
    \item we first form the product $U_2 = [A]^T_{ {I}} [A]_{ {I}}$, and then we compute $\Theta = D^{-1} U_2 D^{-1}$.
\end{enumerate}
The former approach requires $m\times n$ additional memory to store $U_1$ and it requires $m^2k_{\max}^2$ flops to compute $U_1$ and $(2m-1)k_{\max}(k_{\max}-1)/2$ flops for the upper triangular part of $ U_1^T U_1$, while the latter does not require additional memory since the matrix $U_2$ can be stored in the same memory space used for the cosine matrix $\Theta$, and it requires $(2m-1)k_{\max}(k_{\max}-1)/2$ flops the upper triangular part of $U_2$ and $k_{\max}(k_{\max}-1)$ flops for the upper triangular part of $D^{-1} U_2 D^{-1}$. Therefore, we recommend the second approach, even if it requires to write an \emph{ad hoc} low level routine which is not implemented in the BLAS library. 

In order to limit the cost and the amount of additional memory of Algorithm \ref{alg:DM}, we propose a restricted version of the Deviation Maximization pivoting. If the candidate is given by ${I} = \{j_l: l = 1,\dots , {k_{\max}}\}$, we limit its cardinality to be smaller or equal to a machine dependent parameter $k_{DM}$, that is
\begin{equation}
    {I} = \{j_l : l = 1,\dots , \min({k_{\max}},{k_{DM}})\}.
    \label{eq:candidate_set_limited_by_kDM}
\end{equation}
We refer to the value $k_{DM}$ as block size, and we discuss its value in terms of achieved performances in Section \ref{sec:numerical_experiments}.

Notice that the Deviation Maximization requires the inversion of the diagonal matrix $D = \diag(\Vert \mathbf{c}_j \Vert)$, with $j\in {I}$ This operation may cause numerical instabilities when $\Vert \mathbf{c}_j \Vert $ is close to the working precision $\epsilon$, which is likely to be true when the decomposition has revealed the numerical rank, i.e. $n_s \geq r$. In such a case, the computation of $D^{-1}$ would be totally inaccurate. Hence, we require 
\begin{equation}
    \max_i \left\Vert  \mathbf{c}_i \right\Vert > \mathcal{O}(\epsilon),
    \label{eq:column_norm_bound_DM}
\end{equation}
in order to carry out the Deviation Maximization procedure. 

We now describe the most crucial aspects of a practical implementation of the rank-revealing QRDM presented in Algorithm \ref{alg:QRDM}. 
First, the Deviation Maximization block pivoting cannot be carried out when the maximum partial column norm of the trailing matrix is of the order of the working precision $\epsilon$, that is when \eqref{eq:column_norm_bound_DM} holds. In such case, a practical implementation switches from the Deviation Maximization block pivoting to another one, e.g. the standard column pivoting.
Let us now detail how to choose $\tau$ and $\delta$.  
In practice, as we detail in Section \ref{sec:numerical_experiments}, it is desirable to relax the requirements given by Lemma \ref{lem:bound_sigma_T} on the choice of the values for $\tau$ and $\delta$, since their theoretical bounds turn out to be very demanding with a consequent limitation of the performance of the overall factorization. On the other side, if we settle for any choice of $ \tau, \delta$ with $1\geq \tau, \delta \geq 0$, then the Deviation Maximization may identify a set of numerically linear dependent columns. In order to overcome this issue, we incorporate an additional check in the Householder procedure at step \ref{alg:step_QRDM_WYrepr} of Algorithm \ref{alg:QRDM}. 
The modified procedure is presented in Algorithm \ref{alg:QRDM_mod}.

\begin{algorithm}[h]
    \caption{Modified QR with Deviation Maximization $QRDM2(A, \tau, \gamma, \varepsilon)$}
    \label{alg:QRDM_mod}
    \begin{algorithmic}[1]
        \State set $n_s=0$ and initialize the vector $\mathbf{u}$ of squared norms 
        \While{$n_s<n$}
        \State $\left\{j_1,\dots,j_{k_s}\right\} = DM([A]_{n_s+1:m,n_s:n}, [\mathbf{u}]_{n_s+1:n},\tau, \gamma)$
        \State move columns $ \left\{j_1,\dots,j_{k_s}\right\}$ to the leading positions of $A$ and $\mathbf{u}$
        \For{$l=1,\dots, k_s$}
        \State compute the Householder reflector $\mathbf{v}^{(n_s+l)}$ w.r.t. $[A]_{n_s+l:m,n_s+l}$
        \State update the remaining columns $[A]_{n_s+l:m,n_s+l+1:n_s+k_s} -= \mathbf{v}^{(n_s+l)}(\mathbf{v}^{(n_s+l)})^T\ [A]_{n_s+l:m,n_s+l+1:n_s+k_s}$
        \If{$l+1<k_s$ and $\Vert [A]_{n_s+l+1:m,n_s+l+1} \Vert < \varepsilon_s $}{ break} \label{alg:modified_houseQR_check}
        \EndIf
        \EndFor
        \State compute the compact WY representation of $\mathbf{v}^{(n_s+1)},\dots, \mathbf{v}^{(n_s+l)}$ 
        \State update the trailing matrix $[A]_{n_s+1:m,n_s+k_s+1:n} -= Y^{(s)}\ (W^{(s)})^T\ (Y^{(s)})^T\ [A]_{n_s+1:m,n_s+k_s+1:n}$
        \State update the partial column norms $\mathbf{u}$
        \State $n_s = n_s+l$
        \EndWhile
    \end{algorithmic}
\end{algorithm}

Recall that the columns chosen by the Deviation Maximization at the $s$-th algorithmic step satisfy 
\begin{equation}
  \Vert [A]_{n_s:m,n_s+j} \Vert \geq \tau\ \max_{i > n_s} \Vert [A]_{n_s:m,i} \Vert  =:  \varepsilon_s, \label{eq:lenght_check}
\end{equation}
for all $j \in {J}$. The prior check introduced at step \ref{alg:modified_houseQR_check} breaks the Householder procedure when a partial column norm $\Vert [A]_{n_s+l:m,n_s+l} \Vert$ becomes smaller than $\varepsilon_s$ defined above, in other words, if the $l$-th column is not sufficiently linearly independent from the subspace spanned by the first $l-1$ columns already processed. Different choices of $\varepsilon_s$ are possible, e.g. a small and constant threshold. However, numerical tests show that the choice \eqref{eq:lenght_check} works well in practice.
In this case, the Householder reduction to triangular form terminates with $l < k$ Householder reflectors, and the algorithm continues with the computation and the application of the compact WY representation of these $l$ reflectors. At the next iteration, the pivoting strategy moves the rejected column from the leading position, if necessary.

As we show in Section \ref{sec:numerical_experiments}, this break mechanism enables us to choose $\tau$ and $\delta$ rather simply and to obtain the best results in execution times.

Last, we discuss the structure of the permutations employed in the QR fectorization, which has a significant impact on the cost of the algorithm. The structure of the column exchanges determines the structure of $\Pi^{(s)}$ and hence the cost of the QR factorization update. Recall at the $s$-th algorithmic step we have to move columns of indices $j_1,\dots,j_{k_s}$ to the leading positions $n_s+1,\dots, n_s+k_s$. We prefer permutations consisting of a sequence of cyclic shifts, that is a cyclic permutation involving only two elements and fixing all the others. In this way, the application of $\Pi^{(s)}$ requires only $m$ additional memory slots, that is the memory needed to swap two columns. Obviously, the less columns to swap the less the work involved in memory communications. A strategy that can easily be implemented consists in swapping the $n_s+i$-th column with the $j_i$-th column, for $i = 1,\dots,k_s$.

\section{Numerical experiments} \label{sec:numerical_experiments}

In this section we discuss the numerical accuracy of QRDM and compare it with QP3 and the SVD decomposition. Particular importance is given to the values on the diagonal of the upper triangular factor $R$ of the RRQR factorization, which are compared with the singular values of the $R_{11}$ block and with the singular values of the input matrix $A$. We use the implementation of QP3 provided by LAPACK's \texttt{xgeqp3} routine. 
The tests are carried out on a subset of matrices from the San Jose State University Singular Matrix Database, which were used in other previous papers on the topic, see e.g. \cite{DGG15, GEi96}. We show results coming from two subsets of this collection, that we call:
\begin{itemize}
    \item ``small matrices", 261 numerically singular matrices with $m \leq 1024$, $n \leq 2048$, whose indices where obtained with the following Matlab pseudocode
    \begin{algorithmic}[1]
    \State ind = SJget;
    \State index = find ( index.ncols$>32$ \& ind.ncols$<=2048$ \& ind.nrows$<=1024$ );
    \State [\texttildelow, k] = sort ( ind.ncols (index) );
    \State index = index(k);
\end{algorithmic}
    \item ``big matrices", 247 numerically singular matrices  with $m > 1024$, $n > 2048$, whose indices where obtained with the following Matlab pseudocode
    \begin{algorithmic}[1]
    \State ind = SJget;
    \State index = find ( index.ncols$>32$ \& ind.ncols$>2048$ \& ind.nrows$>1024$ );
    \State [\texttildelow, k] = sort ( ind.ncols (index) );
    \State index = index(k(1:247));
\end{algorithmic}
\end{itemize}

For each matrix $A$, we denote by $\sigma_i$ the $i$-th singular value of $A$ computed with the LAPACK's \texttt{xgejsv} routine, and by $n_r$ the numerical rank computed with the option \texttt{JOBA='A'}: in this case, small singular values are comparable with round-off noise and the matrix is treated as numerically rank deficient. As the pivoting used in QRDM does not guarantee that the diagonal values of the factor $R$ are monotonically non-increasing in modulus, for each matrix we denote by $d_i$ the $i$-th largest value among the first $n_r$ diagonal entries considered with positive sign.  
The results provided by QP3 for the two collections are summarised in Figures \ref{fig:ratio_diag(R11)_over_sigma_i_QP3} and \ref{fig:optimality_of_R11_QP3}.
We show that the order of magnitude of the ratios $d_i/\sigma_i$ ranges from $10^{-1}$ to $10^{1}$ for the ``small matrices" (Fig. \ref{fig:ratio_diag(R11)_over_sigma_i_QP3_a}) and ``big matrices" collections  (Fig. \ref{fig:ratio_diag(R11)_over_sigma_i_QP3_b}), i.e. the positive diagonal value $d_i$ approximate the corresponding singular value $\sigma_i$ up to a factor $10$, for $i = 1,\leq,n_r$. 
Moreover, we compare $\sigma_i(R_{11})$, that is the $i$-th singular value of $R_{11} = [R]_{1:n_r,1:n_r}$ computed by LAPACK's \texttt{xgejsv}, with $\sigma_i$, namely the corresponding singular value of $A$, by taking into account the ratios  $\sigma_i(R_{11})/\sigma_i$ for the ``small matrices" (Fig. \ref{fig:optimality_of_R11_QP3_b}) and ``big matrices" collections  (Fig. \ref{fig:optimality_of_R11_QP3_b}). These results confirm that QP3 provides an approximation of the singular values up to a factor $10$.
\begin{figure}
    \centering
    \begin{subfigure}{.9\textwidth}
    \includegraphics[width=\textwidth]{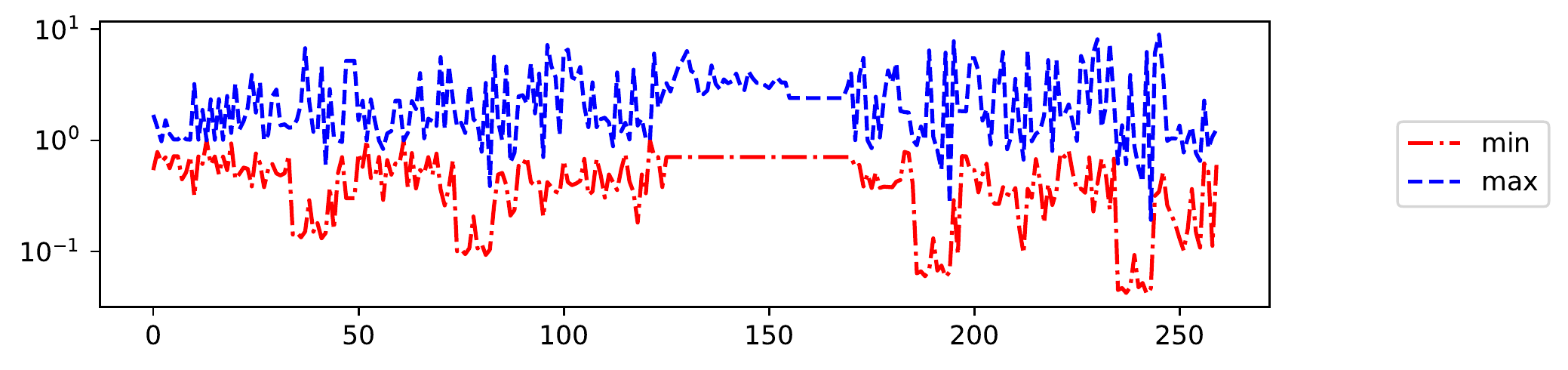}
    \caption{\label{fig:ratio_diag(R11)_over_sigma_i_QP3_a}}
    \end{subfigure}
    \begin{subfigure}{.9\textwidth}
    \includegraphics[width=\textwidth]{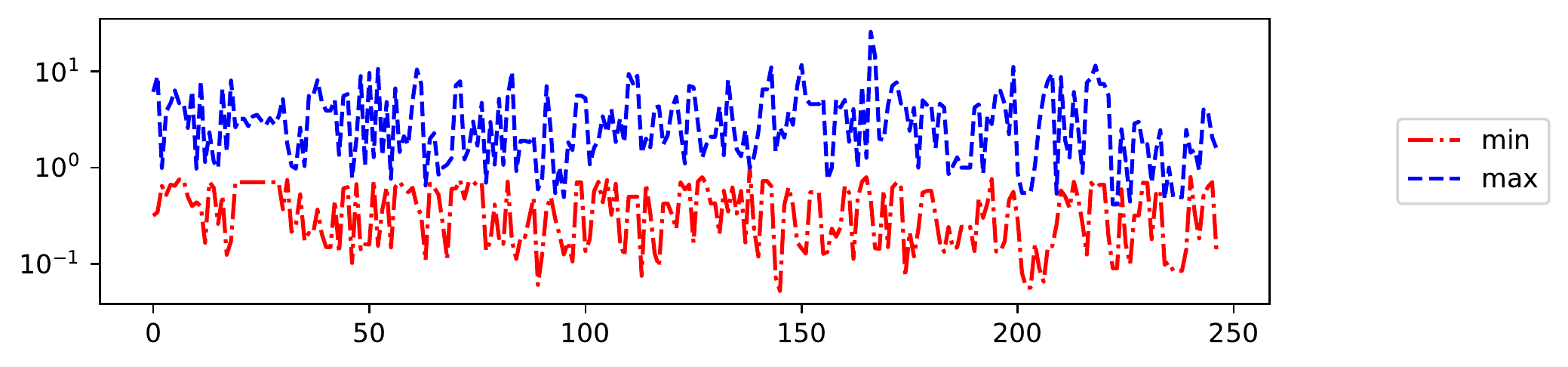}
    \caption{\label{fig:ratio_diag(R11)_over_sigma_i_QP3_b}}
    \end{subfigure}    
    \caption{Ratio $d_i/\sigma_i$, minimum (red) and maximum (blue) values for QP3 on the set ``small matrices" (a) and ``big matrices" (b). \label{fig:ratio_diag(R11)_over_sigma_i_QP3}}
\end{figure}
\begin{figure}
    \centering
    \begin{subfigure}{.9\textwidth}
    \includegraphics[width=\textwidth]{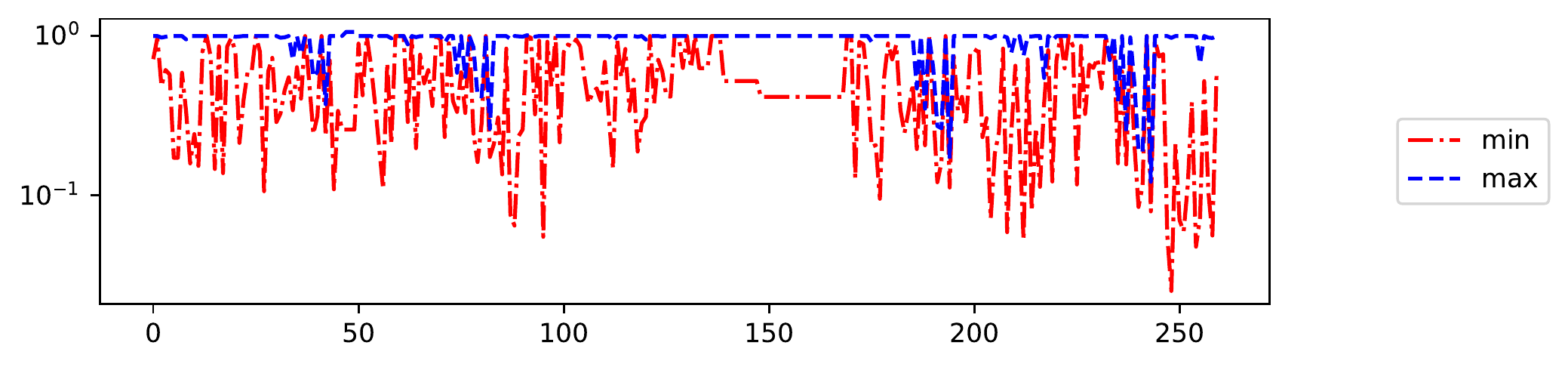}
    \caption{\label{fig:optimality_of_R11_QP3_a}}
    \end{subfigure}
    \begin{subfigure}{.9\textwidth}
    \includegraphics[width=\textwidth]{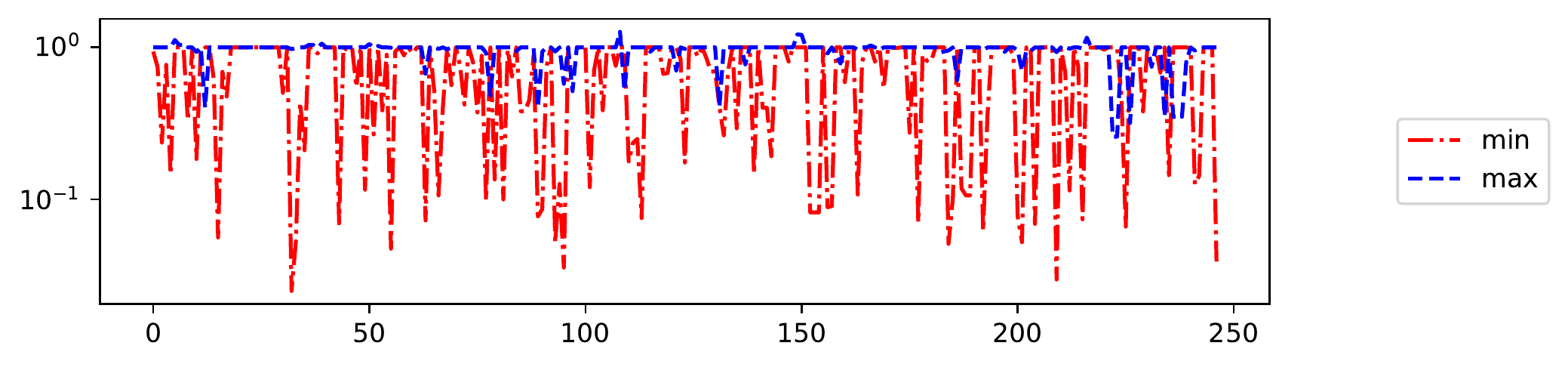}
    \caption{\label{fig:optimality_of_R11_QP3_b}}
    \end{subfigure}    
    \caption{Ratio $\sigma_i(R_{11})/\sigma_i$, minimum (red) and maximum (blue) values for QP3 on the set ``small matrices" (a) and ``big matrices" (b). \label{fig:optimality_of_R11_QP3}}
\end{figure}

Before providing similar results for QRDM, let us discuss the sensitivity of parameters $\tau$ and $\delta$
to the rank-revealing property \eqref{eq:RRproperty_sigmamin}. 
To this aim, we set a grid $\mathcal{G}$ of values $\mathcal{G}(i,j) = (\delta_i,\tau_j) = (0.05\ i,0.05\ j)$, with $i,j=0,\dots,20$, and we consider the $R$ factor obtained by QRDM for each matrix in the ``small matrices" collection and for each choice $\delta=\delta_i$, $\tau =\tau_j$. 

Figure \ref{fig:how_diag(R11)_approximates_singular_values} shows the order of magnitude of the minimum among all matrices of the minimum ratio $\min_{1\leq i\leq n_r}r_i/\sigma_i$  for each grid point of $\mathcal{G}$. We see that the positive diagonal elements provide an approximation up to a factor $10$ of the singular values for a wide range of parameters, corresponding to the light gray region: in practice, it is sufficient to avoid the extreme cases $\tau = 0$ and $\delta = 1$. Indeed, the presence of a wrong diagonal element in $R$, i.e. too small with respect to the corresponding singular value, is avoided thanks to the additional check proposed in \eqref{eq:lenght_check}, which possibly breaks the Householder triangularization.  
This suggests that any choice of $1 \geq \tau > 0$ and $1 >\delta \geq 0$ may lead to a rank-revealing QR decomposition. In this way, even a greedy setting of the algorithm is viable: if $\tau$ is almost zero and $\delta$ near one, the Deviation Maximization pivoting collects as much columns as possible. 
\begin{figure}
    \centering
    \begin{subfigure}{.5\textwidth}
    \includegraphics[width = \textwidth]{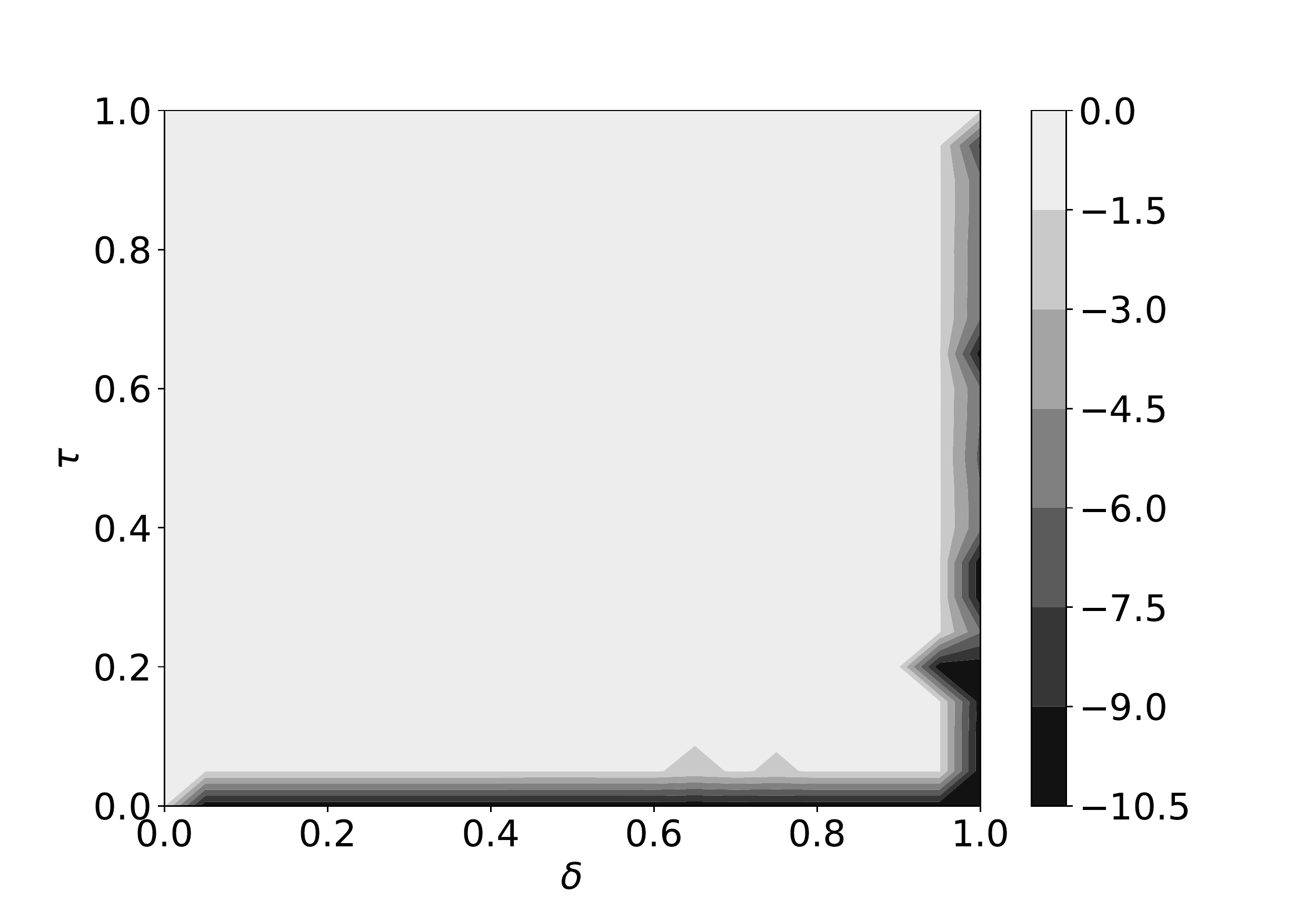}
    \caption{\label{fig:how_diag(R11)_approximates_singular_values}}
    \end{subfigure}
    \begin{subfigure}{.5\textwidth}
    \includegraphics[width = \textwidth]{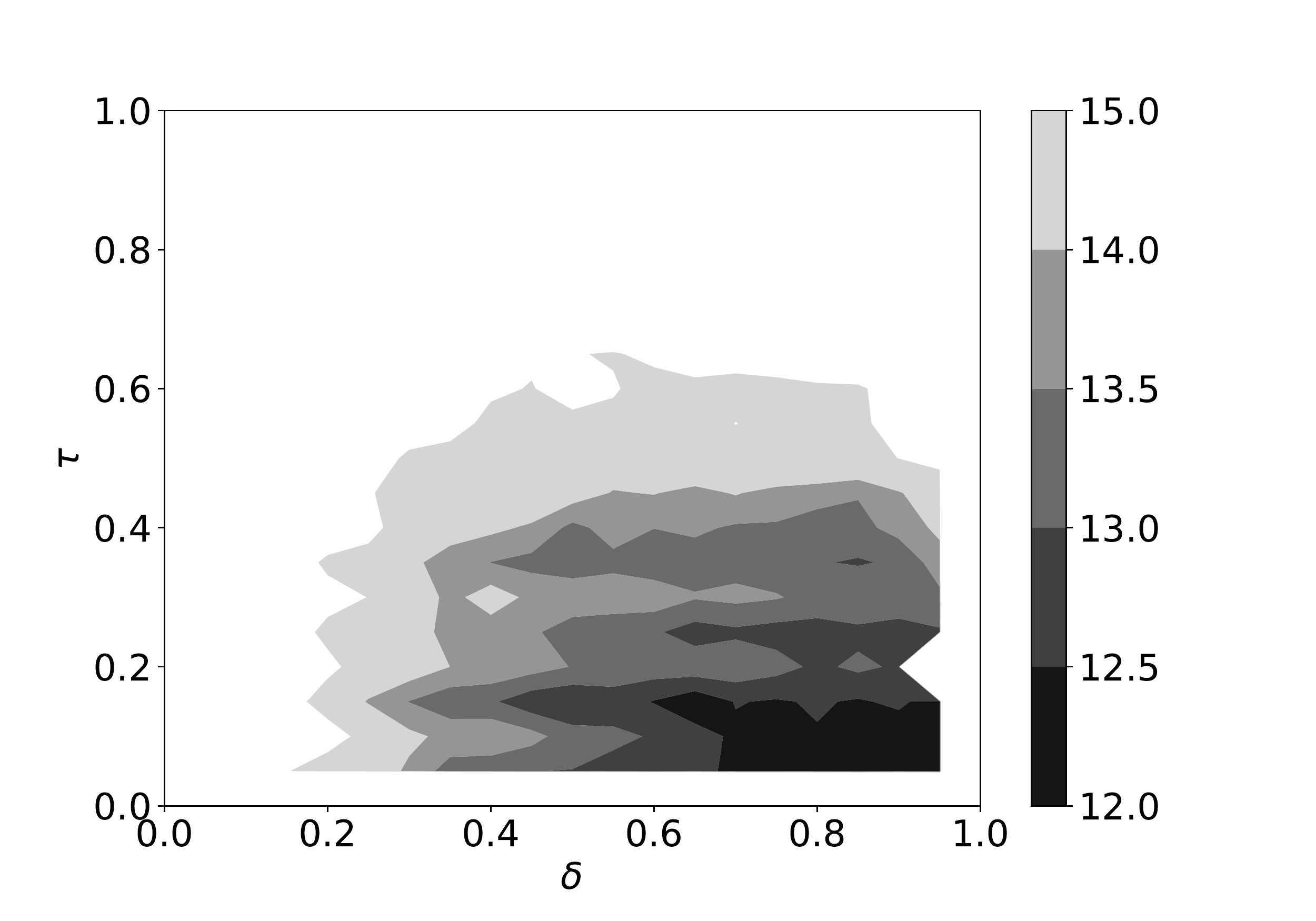}
    \caption{\label{fig:execution_times_for_a_grid_of_parameters}}
    \end{subfigure}
    \caption{Order of magnitude of the minimum $\min(r_i/\sigma_i)$ over all matrices (a) and cumulative execution times for QRDM (b) in function of the parameters $\tau$ and $\delta$ on the ``small matrices" data set.}
\end{figure}
Therefore, for an optimal parameters choice, we look at execution times. 

Figure \ref{fig:execution_times_for_a_grid_of_parameters} shows the cumulative execution times for all tests in the ``small matrices" collection for each grid point of $\mathcal{G}$. It is evident that best performances are obtained toward the right-bottom corner, in correspondence of the dark gray region, confirming that a greedy approach is convenient.
However, a too greedy choice of the parameters may yield less accurate ratios $d_i/\sigma_i$ in a very few cases. Hence, we suggest a safer choice: from now on we set 
$\tau = 0.15$ and $\delta = 0.9$, which are the optimal values for the validation set here considered. 

Figures \ref{fig:ratio_diag(R11)_over_sigma_i_QRDM} and \ref{fig:optimality_of_R11_QRDM} summarize the results provided by QRDM for the two collections with this parameters' choice.
In particular, the positive diagonal entries approximate the singular values up to a factor $10$ for all matrices in the ``small matrices" (Fig. \ref{fig:ratio_diag(R11)_over_sigma_i_QRDM_a}) and ``big matrices" (Fig. \ref{fig:ratio_diag(R11)_over_sigma_i_QRDM_b}) collections, while the singular values of $R_{11}$ provide an approximation up to a factor $10^2$ for few matrices in the ``small matrices" (Fig. \ref{fig:optimality_of_R11_QRDM_a}) and ``big matrices" (Fig. \ref{fig:optimality_of_R11_QRDM_b}) collections.

\begin{figure}
    \centering
    \begin{subfigure}{.9\textwidth}
    \includegraphics[width=\textwidth]{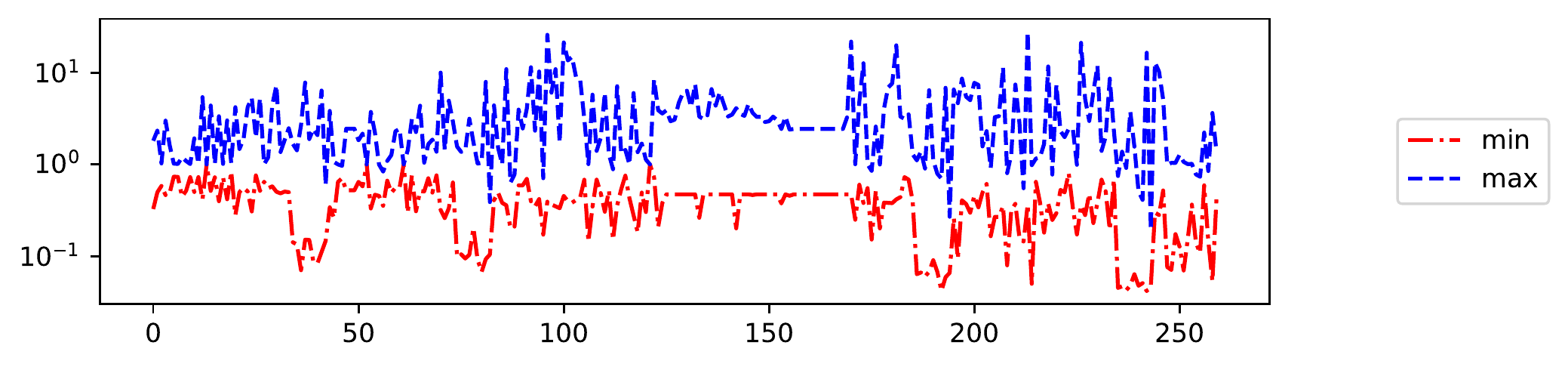}
    \caption{\label{fig:ratio_diag(R11)_over_sigma_i_QRDM_a}}
    \end{subfigure}
    \begin{subfigure}{.9\textwidth}
    \includegraphics[width=\textwidth]{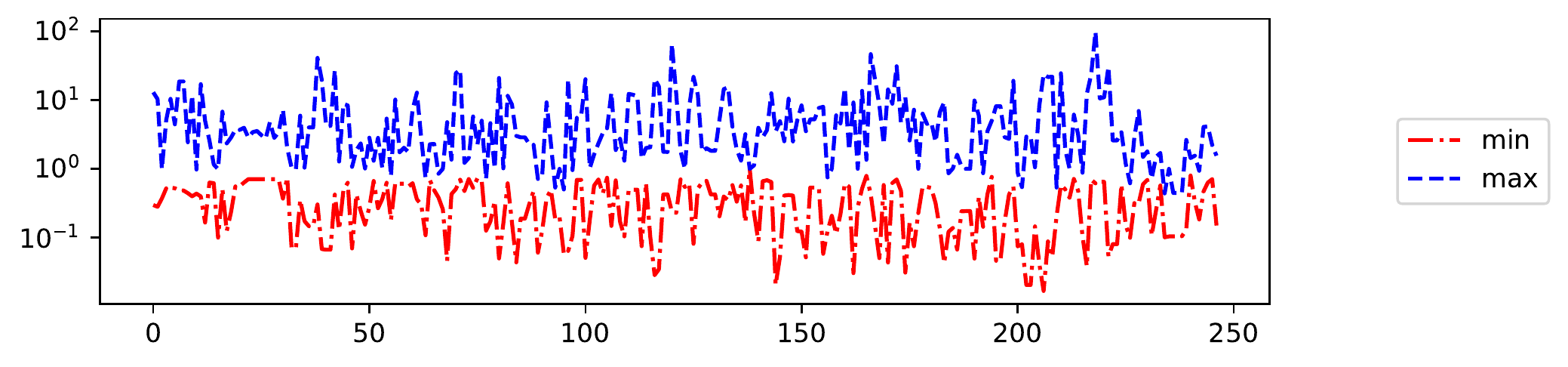}
    \caption{\label{fig:ratio_diag(R11)_over_sigma_i_QRDM_b}}
    \end{subfigure}    
    \caption{Ratio $d_i/\sigma_i$, minimum (red) and maximum (blue) values for QRDM on the set ``small matrices" (a) and ``big matrices" (b). \label{fig:ratio_diag(R11)_over_sigma_i_QRDM}}
\end{figure}

\begin{figure}
    \centering
    \begin{subfigure}{.9\textwidth}
    \includegraphics[width=\textwidth]{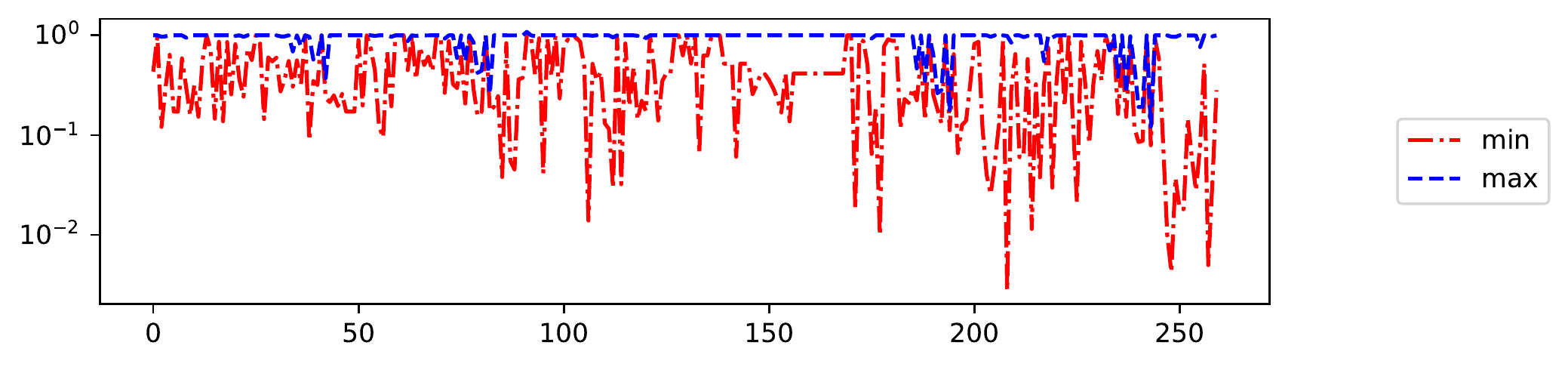}
    \caption{ \label{fig:optimality_of_R11_QRDM_a}}
    \end{subfigure}
    \begin{subfigure}{.9\textwidth}
    \includegraphics[width=\textwidth]{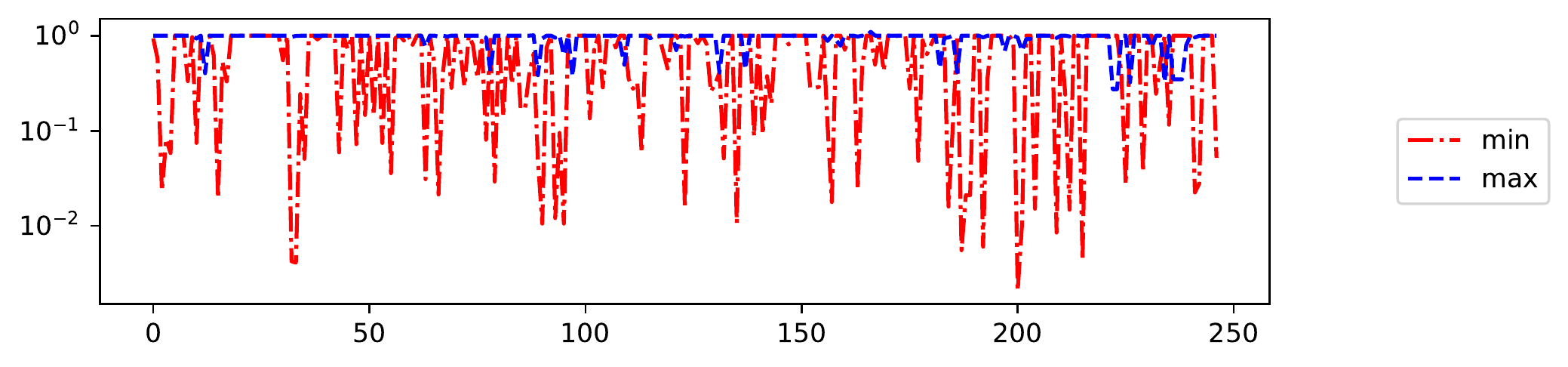}
    \caption{ \label{fig:optimality_of_R11_QRDM_b}}
    \end{subfigure}    
    \caption{Ratio $\sigma_i(R_{11})/\sigma_i$, minimum (red) and maximum (blue) values for QRDM on the set ``small matrices" (a) and ``big matrices" (b). \label{fig:optimality_of_R11_QRDM}}
\end{figure}

Let us now consider QRDM with a stopping criterion. We show the accuracy in the determination of the numerical rank, and the benefits in terms of execution times, when the matrix rank is much smaller than its number of columns. Recall that QRDM switches to the scalar pivoting when the partial column norms are not sufficiently large \eqref{eq:column_norm_bound_DM}, affecting the algorithm's performance. We consider the stopping criterion in \eqref{eq:practical stopping criterion}-\eqref{eq:practical stopping criterion_choice1}: the numerical rank is this case is given by the number of columns processed by QRDM and we denote it by $n_r^{QRDM}$.

Figure \ref{fig:err_rank} shows the ratio $\left( n_r^{QRDM}-n_r \right)/n_r$, where $n_r$ is the number of singular values larger than $  \epsilon n \Vert A \Vert = \epsilon n \sigma_1$, for all matrices in the ``small matrices" (Fig. \ref{fig:err_rank_a}) and ``big matrices" (Fig. \ref{fig:err_rank_a}) collections. The computed rank is accurate in nearly all cases. We also considered the stopping criterion in \eqref{eq:practical stopping criterion} with the choice \eqref{eq:practical stopping criterion_choice2}, which turned out to be less accurate.
Figure \ref{fig:err_rank_a} shows a case in which $n_r^{QRDM}$ overestimates $n_r$ with a relative error of about 45\%. This is a pathological case, however. The matrix involved shows a gap in the singular values distribution and, immediately after, a group of singular values just below the value $\epsilon n \Vert A \Vert$. The corresponding diagonal entries of the matrix $R$ obtained by QRDM show the same gap, but the stopping criterion \eqref{eq:practical stopping criterion}-\eqref{eq:practical stopping criterion_choice1}, which approaches the quantity $\epsilon n \Vert A \Vert$ from below, does not detect so accurately the crossing of the threshold. Anyway, since the diagonal entries of $R$ describe well the corresponding gap in the singular values, even in this case, the applications can correctly truncate the $R$ factor in post-processing and form the corresponding $Q$ factor.

\begin{figure}
    \centering
    \begin{subfigure}{.9\textwidth}
    \includegraphics[width=.9\textwidth]{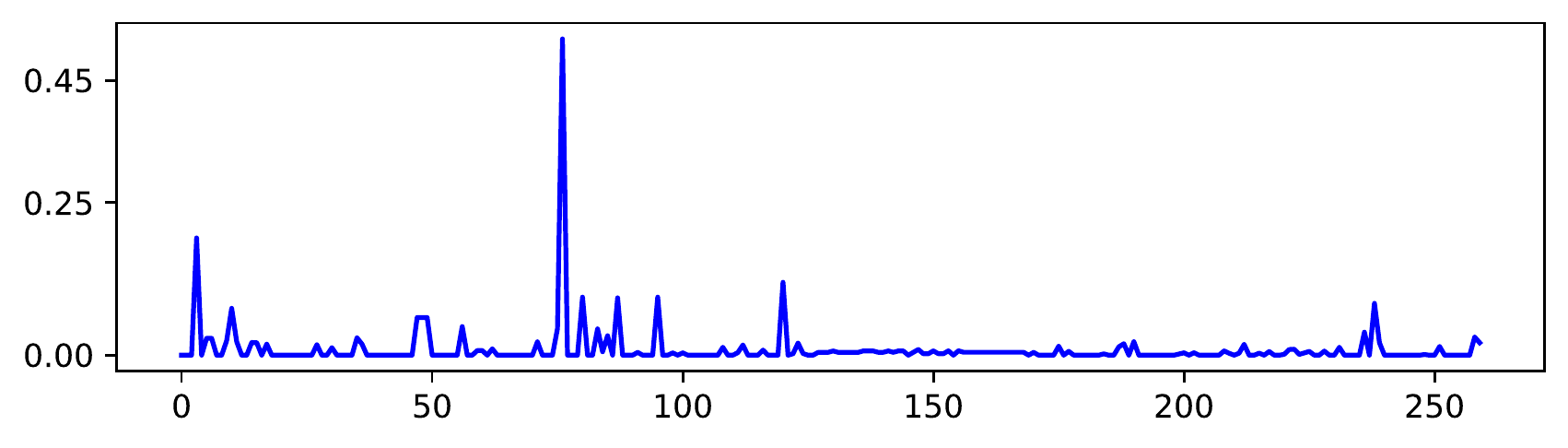}
    \caption{\label{fig:err_rank_a}}
    \end{subfigure}
    \begin{subfigure}{.9\textwidth}
    \includegraphics[width=.9\textwidth]{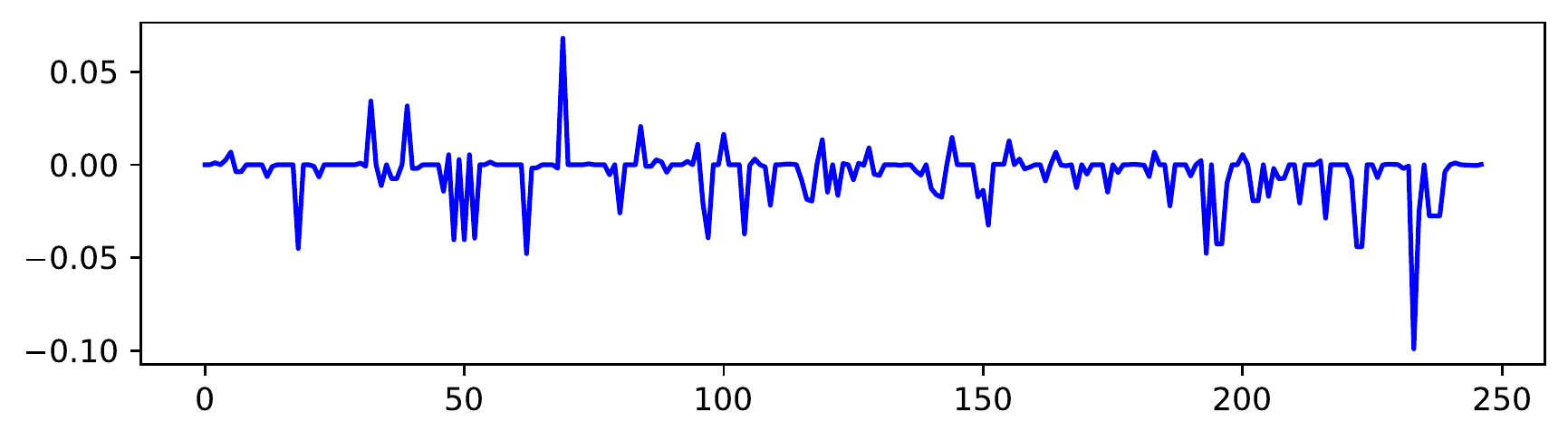}
    \caption{\label{fig:err_rank_b}}
    \end{subfigure}    
    \caption{relative error on the computed numerical rank for QRDM on the set ``small matrices" (a) and ``big matrices" (b). \label{fig:err_rank}}
\end{figure}

\begin{figure}
    \centering
    \includegraphics[width=.7\textwidth]{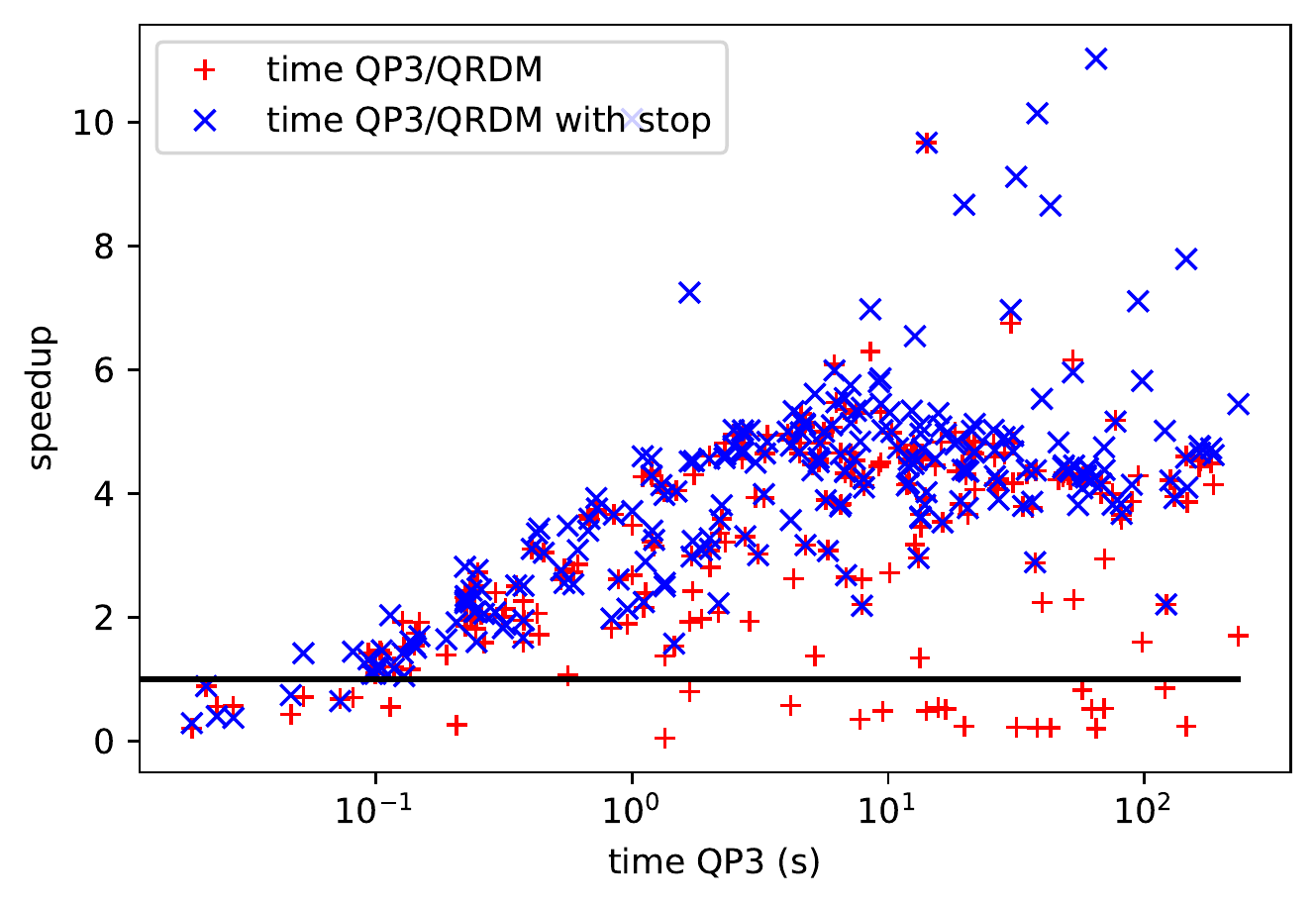}
    \caption{Execution time (s) of QP3 vs the ratio between the execution times of QP3 and QRDM (red) or QRDM with stop criterion (blue). \label{fig:speedup_big}}
\end{figure}

Finally, we compare the performances of QP3 with those of QRDM in terms of their execution times. Figure \ref{fig:speedup_big} shows the speedup of QRDM over QP3, namely the ratio $t_{QP3}/t_{QRDM}$, where $t_{QP3}$ and $t_{QRDM}$ are the execution times (in seconds) of QP3 and QRDM respectively, in function of $t_{QP3}$. We see that QRDM can achieve an average speedup of $4\times$ for medium/large matrices (corresponding to higher execution times). This result is a consequence of a higher BLAS-3 fraction of work provided by QRDM against QP3. 
Indeed, in order to compute the block Householder reflector, QP3 must update the partial column norms and identifies the next pivot by computing the maximum column norm, while QRDM selects a block of pivot columns at once. The former strategy relies on BLAS-2 operations, while the latter mostly on BLAS-3 operations.

Moreover, the stopping criterion gives an additional advantage for matrices whose numerical rank is much smaller than their number of columns. It may also be interesting to consider a comparison with an implementation of QP3 with the same stop criterion, but this is beyond the scope of the present work. 

Last, let us discuss briefly the effect of the block size $k_{DM}$ introduced to limit the cardinality of the candidate set in equation \eqref{eq:candidate_set_limited_by_kDM}. This parameter depends on the specific architecture, mainly in terms of cache-memory size, and typical values are $k_{DM}=32,64,128$. We observed that there is an optimal value of $k_{DM}$, in sense that it gives the smallest for a fixed experimental setting, and its computation is similar to the well-known BLAS block size computation practice, which is out of scope of this paper. For sake of clarity we say that on our personal computer we observed the optimal value $k_{DM}=64$, but other choices gave similar performances, e.g. $k_{DM}=32$.

\section{Conclusions} \label{sec:conclusions} 

In this work we have presented a new subset selection strategy we called  ``Deviation Maximization". Our method relies on cosine evaluation in order to select a subset of sufficiently linearly independent vectors. Despite this strategy is not sufficient by itself to identify a maximal subset of linearly independent columns for a given numerically rank deficient matrix, it can be adopted as a column pivoting strategy. In this work we introduced the Rank-Revealing QR factorization with Deviation Maximization pivoting, briefly called QRDM, and we compared it with the Rank-Revealing QR factorization with standard column pivoting, briefly QRP. We have provided a theoretical worst case bound on the smallest singular value for QRDM and we have shown it is similar to available results for QRP. Extensive numerical experiments confirmed that QRDM reveals the rank similarly to QRP and provides a good approximation of the singular values obtained with LAPACK's \texttt{xgejsv} routine. Moreover, we have shown that QRDM has better execution times than those of the BLAS-3 version of QRP implemented in LAPACK's \texttt{xgeqp3} routine in a large number of test cases.
The software implementation of QRDM used in this article is available at the URL: \url{https://github.com/mdessole/qrdm}. 

Our future work will focus on applying the Deviation Maximization as pivoting strategy to other problems which require column selection, e.g. constrained optimization problems, on which the authors successfully experimented a preliminary version in the context of active set methods for NonNegative Least Squares problems, see \cite{DMV20DRNA, DMV20Ma}. 

\bibliographystyle{abbrvnat}
\bibliography{references}

\begin{thebibliography}{28}
\providecommand{\natexlab}[1]{#1}
\providecommand{\url}[1]{\texttt{#1}}
\expandafter\ifx\csname urlstyle\endcsname\relax
  \providecommand{\doi}[1]{doi: #1}\else
  \providecommand{\doi}{doi: \begingroup \urlstyle{rm}\Url}\fi

\bibitem[Anderson et~al.(1999)Anderson, Bai, Bischof, Blackford, Demmel,
  Dongarra, Du~Croz, Greenbaum, Hammarling, McKenney, and Sorensen]{lapack99}
E.~Anderson, Z.~Bai, C.~Bischof, S.~Blackford, J.~Demmel, J.~Dongarra,
  J.~Du~Croz, A.~Greenbaum, S.~Hammarling, A.~McKenney, and D.~Sorensen.
\newblock \emph{{LAPACK} Users' Guide}.
\newblock Society for Industrial and Applied Mathematics, Philadelphia, PA,
  third edition, 1999.
\newblock ISBN 0-89871-447-8 (paperback).

\bibitem[Bischof and Hansen(1992)]{BHa92}
C.~Bischof and P.~Hansen.
\newblock {A block algorithm for computing rank-revealing QR factorizations}.
\newblock \emph{Numerical Algorithms}, 2:\penalty0 371--391, 10 1992.
\newblock \doi{10.1007/BF02139475}.

\bibitem[Bischof and Quintana-Ort\'{i}(1998{\natexlab{a}})]{BQu98}
C.~Bischof and G.~Quintana-Ort\'{i}.
\newblock {Computing Rank-Revealing QR Factorizations of Dense Matrices}.
\newblock \emph{ACM Trans. Math. Softw.}, 24:\penalty0 226--253, 06
  1998{\natexlab{a}}.
\newblock \doi{10.1145/290200.287637}.

\bibitem[Bischof and Quintana-Ort\'{i}(1998{\natexlab{b}})]{BQu98b}
C.~Bischof and G.~Quintana-Ort\'{i}.
\newblock {Algorithm 782: Codes for Rank-Revealing QR Factorizations of Dense
  Matrices}.
\newblock \emph{ACM Transactions on Mathematical Software}, 24:\penalty0
  254--257, 07 1998{\natexlab{b}}.
\newblock \doi{10.1145/290200.287638}.

\bibitem[{Bischof}(1989)]{Bis89}
J.~R. {Bischof}.
\newblock {A block QR factorization algorithm using restricted pivoting}.
\newblock In \emph{Supercomputing '89:Proceedings of the 1989 ACM/IEEE
  Conference on Supercomputing}, pages 248--256, 1989.
\newblock \doi{10.1145/76263.76290}.

\bibitem[Businger and Golub(1965)]{BGo65}
P.~Businger and G.~H. Golub.
\newblock {Linear Least Squares Solutions by Householder Transformations}.
\newblock \emph{Numer. Math.}, 7\penalty0 (3):\penalty0 269–276, June 1965.
\newblock ISSN 0029-599X.
\newblock \doi{10.1007/BF01436084}.

\bibitem[Chan(1987)]{Cha87}
T.~F. Chan.
\newblock {Rank revealing QR factorizations}.
\newblock \emph{Linear Algebra and its Applications}, 88-89:\penalty0 67 -- 82,
  1987.
\newblock ISSN 0024-3795.
\newblock \doi{https://doi.org/10.1016/0024-3795(87)90103-0}.

\bibitem[Chandrasekaran and Ipsen(1994)]{CIp94}
S.~Chandrasekaran and I.~C.~F. Ipsen.
\newblock {On Rank-Revealing Factorisations}.
\newblock \emph{SIAM Journal on Matrix Analysis and Applications}, 15\penalty0
  (2):\penalty0 592--622, 1994.
\newblock \doi{10.1137/S0895479891223781}.

\bibitem[Demmel et~al.(2015)Demmel, Grigori, Gu, and Xiang]{DGG15}
J.~Demmel, L.~Grigori, M.~Gu, and H.~Xiang.
\newblock {Communication Avoiding Rank Revealing QR Factorization with Column
  Pivoting}.
\newblock \emph{SIAM Journal on Matrix Analysis and Applications}, 36:\penalty0
  55--89, 01 2015.
\newblock \doi{10.1137/13092157X}.

\bibitem[Dessole et~al.(2020{\natexlab{a}})Dessole, Marcuzzi, and
  Vianello]{DMV20DRNA}
M.~Dessole, F.~Marcuzzi, and M.~Vianello.
\newblock {Accelerating the Lawson-Hanson NNLS solver for large-scale
  Tchakaloff regression designs}.
\newblock \emph{Dolomites Research Notes on Approximation}, 13:\penalty0 20 --
  29, 2020{\natexlab{a}}.
\newblock ISSN 2035-6803.
\newblock \doi{http://dx.doi.org/10.14658/PUPJ-DRNA-2020-1-3}.

\bibitem[Dessole et~al.(2020{\natexlab{b}})Dessole, Marcuzzi, and
  Vianello]{DMV20Ma}
M.~Dessole, F.~Marcuzzi, and M.~Vianello.
\newblock {dCATCH---A Numerical Package for d-Variate Near G-Optimal Tchakaloff
  Regression via Fast NNLS}.
\newblock \emph{Mathematics}, 8, 7 2020{\natexlab{b}}.
\newblock \doi{https://doi.org/10.3390/math8071122}.

\bibitem[Drma\v{c} and Bujanovi\'{c}(2008)]{DBu08}
Z.~Drma\v{c} and Z.~Bujanovi\'{c}.
\newblock {On the Failure of Rank-Revealing QR Factorization Software -- A Case
  Study}.
\newblock \emph{ACM Trans. Math. Softw.}, 35\penalty0 (2), July 2008.
\newblock ISSN 0098-3500.
\newblock \doi{10.1145/1377612.1377616}.

\bibitem[Duersch and Gu(2017)]{DGu17}
J.~A. Duersch and M.~Gu.
\newblock {Randomized QR with Column Pivoting}.
\newblock \emph{SIAM Journal on Scientific Computing}, 39\penalty0
  (4):\penalty0 C263--C291, 2017.
\newblock \doi{10.1137/15M1044680}.

\bibitem[Foster(1986)]{Fos86}
L.~V. Foster.
\newblock {Rank and null space calculations using matrix decomposition without
  column interchanges}.
\newblock \emph{Linear Algebra and its Applications}, 74:\penalty0 47--71,
  1986.
\newblock ISSN 0024-3795.
\newblock \doi{https://doi.org/10.1016/0024-3795(86)90115-1}.

\bibitem[Golub(1965)]{Gol65}
G.~Golub.
\newblock {Numerical Methods for Solving Linear Least Squares Problems}.
\newblock \emph{Numer. Math.}, 7\penalty0 (3):\penalty0 206–216, June 1965.
\newblock ISSN 0029-599X.
\newblock \doi{10.1007/BF01436075}.

\bibitem[Golub and Van~Loan(2013)]{golub-libro}
G.~Golub and C.~Van~Loan.
\newblock \emph{{Matrix Computations (4th ed.)}}.
\newblock Johns Hopkins Studies in the Mathematical Sciences. Johns Hopkins
  University Press, 2013.
\newblock ISBN 9781421407944.

\bibitem[Golub et~al.(1976)Golub, Klema, and Stewart]{GKS76}
G.~Golub, V.~Klema, and G.~W. Stewart.
\newblock {Rank degeneracy and least squares problems}.
\newblock Technical Report STAN-CS-76-559, Department of Computer Science,
  Stanford University, 1976.

\bibitem[Gu and Eisenstat(1996)]{GEi96}
M.~Gu and S.~C. Eisenstat.
\newblock {Efficient Algorithms for Computing a Strong Rank-Revealing QR
  Factorization}.
\newblock \emph{SIAM Journal on Scientific Computing}, 17\penalty0
  (4):\penalty0 848--869, 1996.
\newblock \doi{10.1137/0917055}.

\bibitem[Hansen(1999)]{Han99}
P.~C. Hansen.
\newblock \emph{{Rank-Deficient and Discrete Ill-Posed Problems: Numerical
  Aspects of Linear Inversion}}.
\newblock Society for Industrial and Applied Mathematics, USA, 1999.
\newblock ISBN 0898714036.

\bibitem[Hong and Pan(1992)]{HPa92}
Y.~P. Hong and C.-T. Pan.
\newblock {Rank-Revealing QR Factorizations and the Singular Value
  Decomposition}.
\newblock \emph{Mathematics of Computation}, 58\penalty0 (197):\penalty0
  213--232, 1992.
\newblock ISSN 00255718, 10886842.

\bibitem[Kahan(1966)]{Kah66}
W.~Kahan.
\newblock {Numerical linear algebra}.
\newblock \emph{Canadian Mathematical Bulletin}, 9:\penalty0 757--801, 1966.

\bibitem[Martinsson(2015)]{Mar15}
P.~Martinsson.
\newblock {Blocked rank-revealing QR factorizations: How randomized sampling
  can be used to avoid single-vector pivoting}.
\newblock 05 2015.

\bibitem[Quintana-Ortí et~al.(1998)Quintana-Ortí, Sun, and Bischof]{QSB98}
G.~Quintana-Ortí, X.~Sun, and C.~H. Bischof.
\newblock {A BLAS-3 Version of the QR Factorization with Column Pivoting}.
\newblock \emph{SIAM Journal on Scientific Computing}, 19\penalty0
  (5):\penalty0 1486--1494, 1998.
\newblock \doi{10.1137/S1064827595296732}.

\bibitem[Radons(2016)]{Rad16}
M.~Radons.
\newblock {Direct solution of piecewise linear systems}.
\newblock \emph{Theoretical Computer Science}, 626:\penalty0 97--109, 2016.
\newblock ISSN 0304-3975.
\newblock \doi{https://doi.org/10.1016/j.tcs.2016.02.009}.

\bibitem[Schreiber and VanLoan(1989)]{SVL89}
R.~Schreiber and C.~VanLoan.
\newblock {A Storage-Efficient WY Representation for Products of Householder
  Transformations}.
\newblock \emph{SIAM Journal on Scientific and Statistical Computing}, 10, 02
  1989.
\newblock \doi{10.1137/0910005}.

\bibitem[Thompson(1972)]{Tho72}
R.~Thompson.
\newblock {Principal submatrices IX: Interlacing inequalities for singular
  values of submatrices}.
\newblock \emph{Linear Algebra and its Applications}, 5\penalty0 (1):\penalty0
  1--12, 1972.
\newblock ISSN 0024-3795.
\newblock \doi{https://doi.org/10.1016/0024-3795(72)90013-4}.

\bibitem[Varah(1975)]{VARAH19753}
J.~Varah.
\newblock {A lower bound for the smallest singular value of a matrix}.
\newblock \emph{Linear Algebra and its Applications}, 11\penalty0 (1):\penalty0
  3 -- 5, 1975.
\newblock ISSN 0024-3795.
\newblock \doi{https://doi.org/10.1016/0024-3795(75)90112-3}.

\bibitem[Xiao et~al.(2017)Xiao, Gu, and Langou]{XGL17}
J.~Xiao, M.~Gu, and J.~Langou.
\newblock {Fast Parallel Randomized QR with Column Pivoting Algorithms for
  Reliable Low-Rank Matrix Approximations}.
\newblock pages 233--242, 12 2017.
\newblock \doi{10.1109/HiPC.2017.00035}.

\end{thebibliography}

\appendix

\section{Auxiliary results} \label{bounds}
For sake of completeness, let us list in this appendix some useful facts we often used in this work. In order to help the reader, some results are stated together with the proof, others are simply reported and referenced.

\subsection{About Singular Values}
Let $A$ be an $m \times n$ matrix, and recall that the singular values of $A$ are the roots of the largest $\min(m,n)$ eigenvalues of $A^T A$ or $AA^T$. This is quite evident using the SVD decomposition $A=U\Sigma V^T$, where $U$ and $V$ are unitary matrices of order $m$ and $n$ respectively, and $\Sigma$ is an $m\times n$ pseudo-diagonal matrix (its extra-diagonal elements are null). Since $A^TA = V \Sigma^T \Sigma V^T$ and $AA^T = U \Sigma \Sigma^T U^T$, where $ \Sigma^T \Sigma$ and $ \Sigma \Sigma^T$ are diagonal matrices of order $n$ and $m$ respectively, but they clearly share the same diagonal elements up to index $\min(m,n)$. 
For any orthogonal matrix $Q$ of order $m$, we have
\begin{equation}
    A^TA = A^T Q^T QA = (QA)^TQA,\label{eq:sv_invariance_leftmult}
\end{equation}
therefore the singular values of $A$ and those of $QA$ are equal. On the other hand, if $Q$ is an orthogonal matrix of order $n$, we have
\begin{equation}
    AA^T = A QQ^T A = AQ(AQ)^T.\label{eq:sv_invariance_rightmult}
\end{equation}
This holds in particular for any permutation matrix $\Pi$, hence column or row permutations do not change the singular values of a matrix.
We also have
\begin{equation}
  \left( A^T\ \mathbb{O}^T \right) \left( \begin{array}{c} A \\ \mathbb{O} \end{array} \right) =  \left( \mathbb{O}^T\ A^T \right) \left( \begin{array}{c} \mathbb{O} \\ A \end{array} \right) =  A^TA +  \mathbb{O}^T \mathbb{O} = A^TA, \label{eq:sv_invariance_zeroblock}
\end{equation}
hence the singular values of a matrix do not change if we add a null block of rows or columns to a matrix $A$.

Let us now list and prove some inequalities involving the $2$-norm of a matrix $A =(\mathbf{a}_1 \dots \mathbf{a}_n) $.
\begin{lem}
For any matrix $A$ we have
\begin{equation}
     \max_{i} \Vert \mathbf{a}_i \Vert_2 \leq \Vert A \Vert_2 \leq \sqrt{n} \max_{i} \Vert \mathbf{a}_i \Vert_2. \label{eq:bound_largest_sv}
\end{equation}
\end{lem}
\begin{proof}
Let $\mathbf{e}_i$ be the $i$-th element of the canonical basis of $\mathbb{R}^n$. Then $A\mathbf{e}_i = \mathbf{a}_i$, and the left-hand inequality is proved. For the right-hand inequality, consider $\mathbf{x} \in \mathbb{R}^n$, then
\begin{equation*}
    A\mathbf{x} = \sum x_i \mathbf{a}_i \Rightarrow  \Vert A\mathbf{x} \Vert_2 \leq \sum |x_i| \Vert \mathbf{a}_i \Vert_2.
\end{equation*}
Apply Cauchy-Schwarz inequality and take $\Vert \mathbf{x} \Vert = 1$ to conclude
\begin{equation*}
   \Vert A\mathbf{x} \Vert_2 \leq \Vert \mathbf{x} \Vert_2 \sqrt{ \sum_i \Vert \mathbf{a}_i \Vert^2_2} \leq \sqrt{n} \max_i \Vert \mathbf{a}_i \Vert_2.
\end{equation*}
\qed\end{proof}
The followings are easy consequences of the result above.  
\begin{cor}
For any matrix $A$ we have
\begin{equation}
\Vert A \Vert_{\max} \leq \Vert A \Vert_{2} \leq \sqrt{mn}\Vert A \Vert_{\max}.  \label{eq:bound_max_norm}
\end{equation}
\end{cor}
\begin{proof}
Let $\mathbf{a}_i$ be the $i$-th column of $A$. Then we have
\begin{align*}
    \Vert  \mathbf{a}_i \Vert & = \sqrt{a_{i1}^2 + \dots + a_{im}^2} \leq \sqrt{m} \max_j \sqrt{a_{ij}^2} = \sqrt{m} \max_j  \vert a_{ij} \vert, \\
    \end{align*}
and, for any $1 \leq j \leq m$, we have
\begin{align*}
    \Vert  \mathbf{a}_i \Vert &\geq \sqrt{a_{ij}^2} = \vert a_{ij} \vert.
\end{align*}
Apply these inequalities to \eqref{eq:bound_largest_sv} to conclude.
\qed\end{proof}
\begin{cor}
If $A$ is a nonsingular and its inverse is partitioned into rows as
\begin{equation*}
    A^{-1} = \left( \begin{array}{c}
          \mathbf{b}_1^T \\
          \vdots \\
          \mathbf{b}_n^T
    \end{array} \right),
\end{equation*}
then 
\begin{equation}
    \sigma_{\min}(A) \leq \min_{i} (\Vert \mathbf{b}_i \Vert^{-1}_2) \leq \sqrt{n} \sigma_{\min}(A). \label{eq:bound_smallest_sv}
\end{equation}
\end{cor}
\begin{proof} The Lemma above applied to $A^{-T}$ yields
\begin{equation}
     \max_{i} \Vert \mathbf{b}_i \Vert_2 \leq \Vert A^{-T} \Vert_2 \leq \sqrt{n} \max_{i} \Vert \mathbf{b}_i \Vert_2,
\end{equation}
from which we deduce the left-hand inequality
\begin{equation}
     \min_{i} (\Vert \mathbf{b}_i \Vert^{-1}_2) \geq \frac{1}{\Vert A^{-T} \Vert_2} = \frac{1}{\Vert A^{-1} \Vert_2} = \sigma_{\min}(A).
\end{equation}
For the right-hand inequality, consider $\mathbf{x} \in \mathbb{R}^n$, then
\begin{equation*}
    \Vert A^{-1} \mathbf{x} \Vert^2_2 = \left\Vert \begin{array}{c}
         \mathbf{b}_1^T\mathbf{x}  \\
         \vdots \\
         \mathbf{b}_n^T\mathbf{x} 
    \end{array} \right\Vert^2_2 = \sum_i \left(\mathbf{b}_i^T\mathbf{x}\right)^2 \leq \sum_i \left\Vert \mathbf{b}_i \right\Vert^2_2 \left\Vert \mathbf{x}\right\Vert^2_2,
\end{equation*}
where we used Cauchy-Schwarz inequality. We have
\begin{equation*}
    \Vert A^{-1}\Vert^2_2 = \max_{\Vert \mathbf{x} \Vert_2 = 1} \Vert A^{-1}  \mathbf{x} \Vert^2_2 \leq  \max_{\Vert \mathbf{x} \Vert_2 = 1} \sum_i \left\Vert \mathbf{b}_i \right\Vert^2_2 \left\Vert \mathbf{x}\right\Vert^2_2 =
    \sum_i \left\Vert \mathbf{b}_i \right\Vert^2_2 \leq n  \max_i \left\Vert \mathbf{b}_i \right\Vert^2_2,
\end{equation*}
from which we deduce
\begin{equation*}
    \min_{i} (\Vert \mathbf{b}_i \Vert^{-1}_2) \geq \frac{\sqrt{n}}{\Vert A^{-1}\Vert_2} = \sqrt{n} \sigma_{\min}(A).
\end{equation*}
\qed\end{proof}

\subsection{About Strictly Diagonally Dominant matrices}
A matrix $A$ is said to be Strictly Diagonally Dominant (SDD) by rows if
\begin{equation*}
    | a_{ii}| > \sum_{j\neq i} | a_{ij}|, 
\end{equation*}
for all $i$. We say $A$ is SDD by columns if $A^T$ is SSD by rows.
The following result is taken from \cite{Rad16}.
\begin{lem} \label{lem:inverse_SDD}
Let $\Theta = \mathbb{I} - S$, with $\Vert S \Vert_{\infty} < \frac{1}{2}$. Then $\Theta^{-1}$ exists, it has a positive diagonal and it is strictly diagonally dominant.

\end{lem}
\begin{proof}
In this case Neumann series converges, and we have
\begin{equation}
    \bar{\Theta} = {\Theta}^{-1} = \sum_{k=0}^{\infty} (\mathbb{I} - \Theta)^k = \sum_{k=0}^{\infty} S^k  = \mathbb{I} + \sum_{k=1}^{\infty} S^k,
    \label{eq:neumann_series}
\end{equation}
hence
\begin{equation*}
    \max_i \sum_j |\bar{\theta}|_{ij} = \left\Vert \bar{\Theta} \right\Vert_\infty = \left\Vert \mathbb{I} + \sum_{k=1}^{\infty} S^k \right\Vert_{\infty} \leq 1+  \left\Vert \sum_{k=1}^{\infty} S^k \right\Vert_{\infty} < 2,
\end{equation*}
since $ \left\Vert \sum_{k=1}^{\infty} S^k \right\Vert_{\infty} \leq \sum_{k=1}^{\infty} \left\Vert S^k \right\Vert_{\infty} \leq \sum_{k=1}^{\infty} \left\Vert S \right\Vert^k_{\infty} < \sum_{k=1}^{\infty} \frac{1}{2}^k = 1$. Moreover, we also have that 
\begin{equation*}
    1 > \left\Vert \sum_{k=1}^{\infty} S^k \right\Vert_{\infty} = \max_i \sum_j \left| \sum_{k=1}^{\infty} S^k_{ij}\right| \geq \sum_j \left| \sum_{k=1}^{\infty} S^k_{ij}\right| \geq \left| \sum_{k=1}^{\infty} S^k_{ij}\right|,
\end{equation*}
for all choices of $i,j$.
Considering eq. \eqref{eq:neumann_series} entrywise, we get
\begin{equation*}
    \bar{\theta}_{ij} = \left( \sum_{k=0}^{\infty} S^k  \right)_{ij}= \mathbb{I}_{ij} + \sum_{k=1}^{\infty} S^k_{ij},
\end{equation*}
implying that $\bar{\theta}_{ii} > 0$. Moreover, we have
\begin{equation*}
    \begin{aligned}
        &\bar{\theta}_{ii} - \sum_{j \neq i} \left| \bar{\theta}_{ij} \right| = 1+ \sum_{k=1}^{\infty} S^k_{ii} - \sum_{j \neq i} \left| \sum_{k=1}^{\infty} S^k_{ij} \right|\\
        & > \sum_{j} \left| \sum_{k=1}^{\infty} S^k_{ij} \right| + \sum_{k=1}^{\infty} S^k_{ii} - \sum_{j \neq i} \left| \sum_{k=1}^{\infty} S^k_{ij} \right| \\
        & = \sum_{j \neq i} \left( \left| \sum_{k=1}^{\infty} S^k_{ij} \right| - \left| \sum_{k=1}^{\infty} S^k_{ij} \right| \right) +
        \left| \sum_{k=1}^{\infty} S^k_{ii} \right| + \sum_{k=1}^{\infty} S^k_{ii} \\
        & = \left| \sum_{k=1}^{\infty} S^k_{ii} \right| + \sum_{k=1}^{\infty} S^k_{ii} \geq 0,
    \end{aligned}
\end{equation*}
therefore $\bar{\Theta}$ is a strictly diagonally dominant matrix with a positive diagonal. 
\qed\end{proof}
\begin{cor} \label{cor:inverse_SDD}
Let $A = \alpha(\mathbb{I} - S)$, with $\alpha>0$ and $\Vert S \Vert_{\infty} < \frac{1}{2}$. Then $A^{-1}$ exists, it has a positive diagonal and it is strictly diagonally dominant.
\end{cor}
\begin{proof}
Apply Lemma \ref{lem:inverse_SDD} to $\mathbb{I} - S$, then $(\mathbb{I} - S)^{-1}$ is SDD with a positive diagonal, and so is $\alpha^{-1}(\mathbb{I} - S)^{-1} = A^{-1}$.
\qed\end{proof}
Let us state some results, for the proof see \cite{VARAH19753}. Let $A$ be SDD by rows,
and set $\alpha = \min_i | a_{ii}| - \sum_{j\neq i} | a_{ij}| > 0$. Then
\begin{equation}
    \Vert  A^{-1} \Vert < \frac{1}{\alpha} \quad \Rightarrow \quad \Vert  A^{-1} \Vert^{-1} = \sigma_{\min}( A) > \alpha.
    \label{eq:bound_sigma_min_SDD}
\end{equation}
If $A$ is SDD both by rows and columns, and $\beta = \min_j |a_{jj}| - \sum_{i\neq j} | a_{ij}| > 0$, then
\begin{equation*}
    \Vert  A^{-1} \Vert^{-1} = \sigma_{\min}( A) \geq \sqrt{\alpha \beta}.
\end{equation*}
If $A$ is block diagonally dominant, i.e.
\begin{equation*}
    A = \left\{ A_{ij}\right\}, 1 \leq i,j\leq n \qquad \Vert A_{ii}\Vert^{-1}_{\infty} > \sum_{j\neq i} \Vert A_{ij}\Vert_{\infty},  
\end{equation*}
and $\alpha = \min_i  \Vert A_{ii}\Vert^{-1}_{\infty} - \sum_{j\neq i} \Vert A_{ij}\Vert_{\infty}$, then
\begin{equation*}
    \Vert  A^{-1} \Vert_{\infty} < \frac{1}{\alpha}.
\end{equation*}

\begin{lem} 
\label{lem:SDD_rescaled_matrix} 
Let $A$ be an $n \times n$ strictly diagonally dominant matrix, with  $\gamma = \min_i 1 - \sum_{j\neq i} | a_{ij}/ a_{ii}| > 0$.
Let $D = \diag(d_1,\dots,d_n)$, and let $1 \geq \tau >0 $ such that $|d_i| \geq \tau \bar{d} >0$, where $ \bar{d} = \max_i |d_i|>0$, for all $i$. The matrix $DAD$ is SDD if $\gamma > 1-\tau^2$.
\end{lem}
\begin{proof}
We have $(DAD)_{ij} = d_id_j a_{ij}$.
For all $1\leq i\leq k$, we have 
\begin{equation*}
    \sum_{j\neq i} |d_id_j a_{ij}| =  |d_i| \sum_{j\neq i} |d_j a_{ij}| \leq |d_i| \bar{d} \sum_{j\neq i} |a_{ij}| \leq  \bar{d}^2 \sum_{j\neq i} |a_{ij}|,
\end{equation*}
and hence
\begin{equation*}
    \begin{aligned}
    & d_i^2 |a_{ii}| - \sum_{j\neq i} | d_id_j a_{ij}| \geq \tau^2\bar{d}^2|a_{ii}| - \bar{d}^2 \sum_{j\neq i} |a_{ij}| \\
    & =\bar{d}^{2} \left( \tau^2 |a_{ii}| + |a_{ii}|(1-\tau^2) - |a_{ii}| (1-\tau^2) - \sum_{j\neq i} |a_{ij}| \right) \\
    & = \bar{d}^2 |a_{ii}|(\gamma - (1-\tau^2)) >0,\\
    \end{aligned}
\end{equation*}
where the last inequality holds iff 
\begin{equation*}
    \gamma > 1-\tau^2.
\end{equation*}
\qed\end{proof}

\end{document}